\documentclass[11pt,a4paper]{amsart}
\usepackage[foot]{amsaddr}

\usepackage{geometry}
\usepackage{amsmath,amssymb,amsthm}
\usepackage{graphicx}
\usepackage{hyperref}
\usepackage{dsfont}

\newtheorem{theorem}{Theorem}[section]
\newtheorem{lemma}[theorem]{Lemma}
\newtheorem{proposition}[theorem]{Proposition}

\newtheorem{assumption}[theorem]{Assumption}
\newtheorem{question}[theorem]{Question}
\theoremstyle{remark}
\newtheorem{remark}[theorem]{Remark}

\numberwithin{equation}{section}

\newcommand{\R}{\mathbb{R}}
\newcommand{\Z}{\mathbb{Z}}
\def\P{\mathbb{P}}
\newcommand{\E}{\mathbb{E}}
\newcommand{\id}{\mathds{1}}

\newcommand{\cross}{\text{\textup{Cross}}}

\begin{document}

\title[Critical level and strong invariance principle for shot noise fields]{Asymptotics for the critical level\\and a strong invariance principle\\for high intensity shot noise fields}
\author{Raphael Lachieze-Rey$^1$}
\address{$^1$Universit\'{e} Paris Descartes / Universit\'e de Paris}
\email{raphael.lachieze-rey@parisdescartes.fr}
\author{Stephen Muirhead$^2$}
\email{smui@unimelb.edu.au}
\address{$^2$School of Mathematics and Statistics, University of Melbourne}
\begin{abstract}
We study fine properties of the convergence of a high intensity shot noise field towards the Gaussian field with the same covariance structure. In particular we (i) establish a strong invariance principle, i.e.\ a quantitative coupling between a high intensity shot noise field and the Gaussian limit such that they are uniformly close on large domains with high probability, and (ii) use this to derive an asymptotic expansion for the critical level above which the excursion sets of the shot noise field percolate.
\end{abstract}
\thanks{}
\date{\today}

\maketitle

\vspace{-0.3cm}


\bigskip

\section{Introduction}

Let $g(x) \in L^1(\R^d)$ be a continuous kernel and let $\mathcal{P}_\lambda$ denote a Poisson point process on $\R^d$ with intensity $\lambda$ with respect to the Lebesgue measure. The \textit{shot noise field with kernel} $g$ is the stationary random field
\[ F_\lambda(x) = \sum_{i \in \mathcal{P}_\lambda} g(x-i) , \]
well-defined almost surely by the integrability of $g$.

\smallskip
We are interested in the global connectivity of the excursion sets
\[ \{ F_\lambda \le \ell \} := \{x \in \R^d : F_\lambda(x) \le \ell  \}  \]
and in particular the range of levels for which the excursion sets have unbounded connected components. By monotonicity there exists a critical level $\ell_c = \ell_c(F_\lambda) \in  [-\infty, \infty]$ such that
\begin{equation}
\label{e:cl}
  \P[ \{ F_\lambda \le \ell \} \text{ contains an unbounded component} ]  =  \begin{cases}  0 & \text{if } \ell < \ell_c,  \\  > 0  & \text{if } \ell > \ell_c, \end{cases}
\end{equation}
 and under mild conditions on $g$ the critical level $\ell_c$ is known to be finite \cite{ms83b, bm17}. However very little is known about the value of $\ell_c$ or how it depends on $\lambda$.
 
 \smallskip
 In this paper we study the asymptotic expansion of $\ell_c(F_\lambda)$ in the `high intensity' limit, i.e.\ as $\lambda \to \infty$. In this limit $F_\lambda$ approximates, after suitable rescaling and centering, the stationary Gaussian field $f$ with the same covariance structure (see Section \ref{s:gaus}). Restricting our attention to the planar case $d=2$, since it is known that $\ell_{c}(f)=0$ for reasons of symmetry and planar duality, one naturally expects that the critical level of the rescaled shot noise field is $\ell_c \approx \ell_{c}(f) = 0$, and our main result gives a quantitative version of this approximation (see Theorem \ref{t:main}).
 
\smallskip Towards this, our second main result is a quantitative coupling between high intensity shot noise fields and the limiting Gaussian field which is strong enough to control the error between their critical levels (see Theorem~\ref{t:sip}). The coupling is valid in all dimensions $d \ge 1$, and the sup norm of the error is presumably optimal (up to logarithmic factors) in dimensions $d=1,2$, in that it corresponds to the rate of convergence of marginals (up to logarithmic factors), with exponential tail decay.
 
 \smallskip This result can be considered a \textit{strong invariance principle} for shot noise fields, which extends existing convergence results available in the literature (see the seminal papers by \cite{Papou,hs85,Lane} about convergence of marginals and functional central limit theorems, and \cite{BieDes12,KluMik} for some generalisations), and we believe it to be of independent interest. First, the proximity between Gaussian fields and high intensity shot noise fields is important in many applied fields (see, e.g., \cite{KluMik} in actuarial sciences, \cite{BieDesEuler,GalGoussMor} in image analysis and texture synthesis, \cite{HeinrichStereo} in stereology, \cite{HaljYan} in the study of wireless networks, and further references therein) and the strong invariance principle gives a means to control this approximation. Second, this result could, in principle, be used to estimate other geometric quantities of shot noise field excursions; see, e.g., the many recent works in probability theory and image analysis on this topic \cite{BieDes12,BieDes12b,BieDesEuler}. Finally, since high intensity shot noise fields are also used as a way to simulate Gaussian fields with the same mean and covariance, this result gives uniform control over the error in the simulation.

\subsection{Asymptotic expansion of the critical level}
As mentioned above, in this paper we consider the asymptotic expansion of $\ell_c(F_\lambda)$ as $\lambda \to \infty$. Most of our attention will be on the planar case $d=2$, although some of the intermediate results hold in all dimensions.

\smallskip To state our main result we introduce some assumptions on the kernel:

\begin{assumption}
\label{a:1}
Suppose $g$ satisfies the following:
\begin{itemize}
\item (Symmetry) $g$ is isotropic;
\item (Smoothness) $g \in C^4(\R^d)$;
\item (Decay) There exist $\beta > d$ such that, for every multi-index $\alpha$ such that $|\alpha| \le 4$,
\[ {  |\partial^\alpha g(x)| \le (1 + |x|)^{-\beta}  \ , \quad x \in \R^d. } \]
\end{itemize}
\end{assumption}

Among other things, the latter two conditions ensure that $F_\lambda$ is almost surely $C^3$-smooth. We use the assumption of isotropy mainly for simplicity, and because it gives us access to results in the literature, although we believe isotropy could be relaxed to symmetry under negation and permutation of coordinate axes with minor adjustments to the arguments.

\smallskip
For technical reasons we shall further assume:

\begin{assumption}
\label{a:3}
The random vector $( F_\lambda(0), \nabla F_\lambda(0) )$ has a bounded density.
\end{assumption}

\noindent We give examples of kernels satisfying Assumption \ref{a:3} in Remark \ref{r:ex} below. A necessary condition for Assumption \ref{a:3} is that $g$ has unbounded support.

\smallskip We can now state our main result, {  which concerns the rescaled and centred shot noise field 
\begin{align*}
f_{\lambda }(x)=\frac{ F_{\lambda }(x)-\mathbb{E}[F_{\lambda }(0)]}{\lambda ^{1/2}} .
\end{align*}
As discussed in detail in the next section, the rescaled field $f_\lambda$ converges as $\lambda \to \infty$ to a limiting stationary Gaussian field~$f$. By the Campbell-Mecke formula $\mathbb{E}[F_\lambda(0)] = \lambda \int_{\R^d} g(x) dx$.}

\smallskip
For functions $j,k : \R^+ \to \R$ we write $j(x) = o(k(x))$ if $|j(x)|/|k(x)| \to 0$ as $x \to \infty$, and $j(x) = O(k(x))$ if there exists a $c > 0$ such that $\limsup_{x \to \infty} |j(x)|/ | k(x)| \le c$. For $h \in L^1(\R^d)$ we write $\int h$ to abbreviate $\int_{\R^d} h(x) dx$. 

\begin{theorem}[Asymptotic expansion of the critical level]
\label{t:main}
Let $d=2$ and assume $g$ satisfies Assumption \ref{a:1} and \ref{a:3}. Then, as $\lambda \to \infty$,
\begin{equation}
\label{e:main1}
\ell_c(f_\lambda) \to 0.
\end{equation}
If moreover $g \ge 0$ then, as $\lambda \to \infty$,
\begin{equation}
\label{e:main2}
\ell_c(f_\lambda) =  O \big( \lambda ^{-1/2}(\log \lambda)^{3/2} \big) .
\end{equation}
\end{theorem}

\begin{remark}
{ By rescaling, \eqref{e:main1} and \eqref{e:main2} are equivalent to the asymptotic expansions $ \ell_c(F_\lambda) = \lambda \int g + o( \lambda^{1/2} ) $ and $ \ell_c(F_\lambda) = \lambda \int g + O( (\log \lambda)^{3/2} )$ respectively.}
\end{remark} 

\begin{remark}[Examples]
\label{r:ex}
For examples of shot noise fields $F_\lambda$ to which Theorem~\ref{t:main} applies, one can take the kernels
\[ g(x) =  (1 + |x|^2)^{-\beta/2}  \, , \  \beta > 2 ,  \quad \text{or} \quad  g(x) = \exp(-(1+ |x|^2)^{-\gamma/2} ) \, , \ \gamma \in (0, 1) . \]
It is clear that these kernels satisfy Assumption \ref{a:1} {  (perhaps after multiplication by a constant)}, and see \cite[Appendix A.2]{lm19} for a proof that the corresponding shot noise fields satisfy Assumption \ref{a:3}.
\end{remark}

Since $f_{\lambda }$ converges to the limiting Gaussian field $f$, and $\ell_{c}(f)=0$, it is natural to expect that {  $\ell_c(f_\lambda) \to 0$}. The main content of Theorem \ref{t:main} is to control the rate of decay of $\ell_c(f_\lambda)$. Nevertheless, we do not expect Theorem \ref{t:main} to be optimal. In fact we believe that {  $\ell_c(f_\lambda)$} is at most ${ O(\lambda ^{-1/2}\log \lambda)}$, and perhaps even smaller, even if we do not assume $g \ge 0$ (see Remark \ref{r:kmt}).

\begin{question}
{ What is the order of $\ell_c(f_\lambda)$ as $\lambda \to \infty$?} How does it depend on the kernel $g$?
\end{question}

\begin{remark}[Possible extensions]
A more general model of shot noise is the random field
\[ F_\lambda(x) = \sum_{i \in \mathcal{P}_\lambda} Y_i g(x-i)  \]
where $\{Y_i\}_{i \in \mathcal{P}_\lambda}$ are independent copies of a random variable $Y$ with finite mean; this reduces to the model we consider if $Y_i \equiv 1$. Under sufficient moment conditions on $Y$ we expect that our proof can be adapted to show that \eqref{e:main1}--\eqref{e:main2} hold. In the case that $Y$ is symmetric then it is expected, and in some cases known (see \cite{lm19}), that { $\ell_c(f_\lambda) = 0$} for all $\lambda > 0$ for reasons of planar duality.
\end{remark}

\subsection{Strong invariance principle for shot noise fields}
\label{s:gaus}

As briefly described above, our proof of Theorem \ref{t:main} relies on the fact that $F_\lambda$ converges, after suitable centring and rescaling, to a stationary Gaussian field $f$.  Indeed a key intermediate step in the proof is to establish a `strong invariance principle' for this convergence, which improves on the qualitative convergence available in the literature.

\smallskip
To state the result we return to the setting of shot noise fields in arbitrary dimension $d \ge 1$ and assume that $g \in L^1(\R^d) \cap L^2(\R^d)$. As above we consider the centred and rescaled shot noise field
\[  f_\lambda(x) = \frac{F_\lambda(x) - \mathbb{E}[ F_\lambda(0)] }{\lambda^{1/2}} . \]
An application of the Campbell-Mecke formula shows that $\E[ F_\lambda(0)] = \lambda \int g$, and also that
\[   \textrm{Cov}[f_\lambda(0), f_\lambda(x) ]   =  \int g(-y) g(x-y) \, dy, \]
which is independent of $\lambda$ and finite by our assumption that $\textrm{Var}[f_\lambda(0)]  = \|g\|_{L^2} < \infty$. Hence it is natural to compare $f_\lambda$ to the stationary centred Gaussian field $f$ with covariance kernel
\[ K(x) = \mathbb{E}[f(0) f(x)] = \int g(-y) g(x-y) \, dy  . \]
Indeed, under a certain subset  of the conditions in Assumption \ref{a:1}, it is known that $f_\lambda$ satisfies a functional central limit theorem (`invariance principle') in the high intensity limit.

\begin{assumption}
\label{a:2}
Suppose $g$ satisfies the following:
\begin{itemize}
\item (Smoothness) $g \in C^2(\R^d)$
\item (Decay) { There exist $\beta > d$ such that, for every multi-index $\alpha$ such that $|\alpha| \le 2$,
\[  |\partial^\alpha g(x)| \le (1+ |x|)^{-\beta} \ , \quad x \in \R^d. \] }
\end{itemize}
\end{assumption}

\noindent In particular these assumptions imply (by dominated convergence) that $K \in C^4(\R^d)$, and that both the shot noise field $f_\lambda$ and the Gaussian field $f$ are almost surely $C^1$-smooth.

\begin{theorem}[Invariance principle for shot noise fields; {\cite[Theorem 8]{hs85}}]
\label{t:qualconv}
Assume $g$ satisfies Assumption \ref{a:2}. Then
\[ f_\lambda \Rightarrow f   \]
in law in the topology of uniform convergence on compact sets.
\end{theorem}

\begin{remark}
The invariance principle in Theorem \ref{t:qualconv} actually holds under slightly weaker conditions (see \cite[Theorem 8]{hs85}) but for simplicity we do not state the most general version.
\end{remark}

To prove our main result (Theorem~\ref{t:main}) we need to quantify the invariance principle in terms of a `coupling distance', that is, we want that $f_\lambda$ and $f$ may be coupled on some sufficiently rich probability space so that they are close with high probability. For $R \ge 1$, let $B(R)$ denote the Euclidean ball of radius $R$ centred at the origin. For $D \subset \R^d$, let $\| \cdot \|_{\infty, D}$ denote the sup-norm on $D$, abbreviating $\|\cdot\|_\infty = \|\cdot\|_{\infty, \R^d}$.

\begin{theorem}[Strong invariance principle for shot noise fields]
\label{t:sip}
$\,$
\begin{enumerate}
\item Assume $g$ satisfies Assumption \ref{a:2} for $\beta > d$. Then there exist $c_1, c_2 > 0$, depending only on $\beta$ and the dimension $d$, such that for every $\lambda \ge 1$ there is a coupling of $f_\lambda$ and $f$ satisfying, for every $R, t \ge 1$,
\[ \P \Big[  \| f_\lambda - f \|_{\infty, B(R)}  \ge  \sigma_d(\lambda) t \Big]  \le c_1 R^d \lambda^{c_{d, \beta}}  e^{-c_2 t} , \]
where
\begin{equation}
\label{e:sigma}
 \sigma_d(\lambda)
= \begin{cases}
\lambda^{-1/2} & d  =1 ,  \\
\lambda^{-1/2}  \sqrt{\log \lambda} &  d = 2 , \\
\lambda^{-1/d} & d \ge 3 ,   \\
\end{cases} \quad \text{and} \quad c_{d, \beta} =  1 + \frac{d}{2}  + \frac{d}{\beta - d}  .
\end{equation}
\item Moreover, suppose $I$ is a finite set of multi-indices, and there is a $\beta > d$ such that each of the partial derivatives $(\partial ^{\alpha }g)_{\alpha \in I}$ satisfy Assumption \ref{a:2} for $\beta$. Then for every $\lambda \ge 1$ there is a coupling of $f_\lambda$ and $f$ satisfying, for every $R, t \ge 1$,
\[ \P \Big[ \exists \alpha \in I: \| \partial ^{\alpha }f_\lambda - \partial ^{\alpha }f \|_{\infty, B(R)}  \ge  \sigma_d(\lambda) t \Big]  \le c_1 | I |  R^d \lambda^{c_{d, \beta}}  e^{-c_2 t} ,\]
where $c_1,c_2 > 0$ depend only on $\beta$ and $d$.
\end{enumerate}
\end{theorem}
\begin{remark}
\label{r:kmt}
The upshot of Theorem \ref{t:sip} is that the convergence of $f_\lambda$ to $f$ takes place on the scale $\sigma_d(\lambda) \log \lambda$, in the sense that one can couple $f_\lambda$ and $f$ so that their difference on a compact set $D \subseteq \R^d$ (or even a polynomially-growing ball $B(\lambda^c)$) is $O(\sigma_d(\lambda)  \log \lambda)$ with high probability. We do not expect this scale to be optimal for all $d \ge 1$, although we believe it to be optimal up to a logarithmic factor if $d \in \{1,2\}$.

\smallskip
To justify the optimality, at least heuristically, recall the celebrated result of Koml\'{o}s-Major-Tusn\'{a}dy (KMT) \cite{kmt75} on coupling the empirical process $\widetilde{\mathcal{E}_n}(s) = n^{-1/2} ( \sum_{i = 1}^n \id_{X_i \le s} - n s )$ on $[0, 1]$, where $X_i$ denote i.i.d.\ random variables uniformly distributed on $[0, 1]$, with a Brownian bridge $\widetilde{\mathcal{W}}_0(s)$. One version of the result states that, for each $n \ge 1$ the coupling can be done so that, for all $t \ge 0$,
\[  \P \Big[  \max_{s \in [0, 1]} |\widetilde{\mathcal{E}_n}(s) - \widetilde{\mathcal{W}}_0(s) |  \ge n^{-1/2} ( c_1 \log n + t ) \Big] \le  c_2 e^{- c_3 t} , \]
where $c_1, c_2, c_3 > 0$ are constants independent of $n$. This implies that the coupling error is at most $O(n^{-1/2} \log n)$ with high probability, and this is thought to be best possible. Further, Beck \cite{bec85} has shown that, in general dimension $d \ge 1$, the analogous strong invariance principle has error that is at least of order $n^{-1/2} (\log n)^{(d-1)/2}$. Hence it is natural to expect that the error in the strong invariance principle for shot noise fields is at least of order $ \lambda^{-1/2} (\log \lambda)^{c_d}$ where $c_d \ge 1$ may depend on the dimension. Indeed one cannot hope to couple $f_\lambda$ and $f$ on any scale smaller than $\lambda^{-1/2}$, since that is the scale on which the marginals converge (by the Berry-Esseen theorem for instance).
\end{remark}

\begin{remark}
Although we are unaware of any result in the literature directly comparable to Theorem \ref{t:sip}, there is a line of related work that has established strong invariance principles for other functionals of the Poisson point process $\mathcal{P}_\lambda$, in particular of the form $(\sum_{i \in \mathcal{P}_\lambda}  \id_{i \in S})_{S \in \mathcal{S}}$, where $\mathcal{S}$ is a class of subsets of $[0, 1]^d$ \cite{cs75, r76a, mass89}. For classes $\mathcal{S}$ that (i) are not too big, and (ii) contain sets with smooth enough boundaries, a strong invariance principle analogous to Theorem \ref{t:sip} has been shown to hold with $\sigma_d(\lambda)$ replaced by $\lambda^{-1/(2d)}$ up to logarithmic factors \cite{mass89}; this result is based on the dyadic coupling scheme of \cite{kmt75}. The fact that the error bound in Theorem \ref{t:sip} is smaller (at least in $d \ge 2$) is due to the smoothness of the kernel $g$ compared to the test functions $\id_{i \in S}$. One can also obtain faster rates (although not as fast as in Theorem \ref{t:sip} in general) by restricting $\mathcal{S}$ to the class of rectangles, which partially compensates for the lack of smoothness (see, e.g., \cite{cs75, tus77, rio96} in the setting of empirical processes).

\smallskip
Interestingly, the works \cite{cs75, r76a, mass89} use the strong invariance principle for the Poisson measure as a means to prove a strong invariance principle for the empirical process via `Poissonisation', whereas our approach follows the reverse pathway, deducing Theorem~\ref{t:sip} from known results on the empirical process due to Koltchinskii \cite{kol94} (and see also \cite{rio94} for related results), which are ultimately based on the dyadic scheme of \cite{kmt75}.
\end{remark}

\begin{remark}
In dimensions $d \ge 3$ it might be possible to improve Theorem \ref{t:sip} by following the alternative approach in \cite{rio96} based on the Strassen-Dudley representation theorem and the multivariate central limit theorems of Zaitsev \cite{zai87a, zai87b} rather than the dyadic coupling scheme of \cite{kmt75}; see also \cite{bm06} for related results. In fact, assuming sufficient smoothness of $g$, we believe that one could achieve the bound
\[ \P \Big[  \| f_\lambda - f \|_{\infty, B(R)}  \ge  \lambda^{-1/2 + \delta} t \Big]  \le c_1 R^d \lambda^{c_2} e^{-c_3 t} , \]
 for arbitrary small $\delta > 0$. Since we are mainly interested in the case $d=2$, and since in $d \in \{1,2\}$ the KMT-based strategy leads to a sharper bound, we do not pursue this here.
\end{remark}

\subsection{Remarks on higher dimensions}
Although our strategy to prove Theorem~\ref{t:main} is essentially planar, since the convergence to a Gaussian limit occurs in all dimensions we believe that if $d \ge 2$
\begin{equation}
\label{e:gend}
 \ell_c(f_\lambda) = \ell_c(f) + O \big(\sigma_d(\lambda) \lambda^{1/2} \log \lambda \big) ,
 \end{equation}
where $\ell_c(f) \in \mathbb{R}$ is the critical level of the limiting Gaussian field $f$ (defined by analogy with~\eqref{e:cl}, and recall that $\ell_c(f)$ depends on the kernel $g$). Only in dimension $d=2$ do we have $\ell_c(f) = 0$ for reasons of planar duality (see Theorem \ref{t:qualzero} below), whereas in general we expect $ \ell_c(f) < 0$ (see \cite{drrv21} for a recent proof for a class of Gaussian fields, and c.f.\ Bernoulli percolation on $\mathbb{Z}^d$, for which the critical parameter satisfies $p_c(\mathbb{Z}^d) < p_c(\mathbb{Z}^2) = 1/2$ for $d \ge 3$ \cite{gr99}). We also do not rule out that $\sigma_d(\lambda) \lambda^{1/2} \log \lambda$ in \eqref{e:gend} could be replaced by $(\log \lambda)^{c_d}$ or even something smaller.

\begin{question}
What is the order of $\ell_c(f_\lambda) - \ell_c$? How does it depend on the kernel $g$ and on the dimension $d$?
\end{question}

{ Although even the qualitative convergence $\ell_c(f_\lambda) \to \ell_c(f)$ remains out of reach  if $d \ge 3$,} using the recent results of Severo \cite{sev21} in place of Theorem \ref{t:qualzero} below, our strategy could be adapted to prove the rigorous lower bound
\begin{equation}
\label{e:gend2}
 \ell_c(f_\lambda) \ge \ell_c(f) + o(1 ) ,
 \end{equation}
 under the conditions in Theorem \ref{t:main} and assuming $g \ge 0$ (the latter condition is needed to apply the results in \cite{sev21}). However the matching upper bound does not follow from our strategy, essentially because one lacks a finite-size criterion for percolation in $d \ge 3$ (see Section \ref{s:fsc}).
 
\subsection{Overview of the rest of the paper}
In Section \ref{s:proof} we prove Theorem \ref{t:main} subject to two intermediate results, namely the strong invariance principle stated in Theorem \ref{t:sip} above, and a continuity criterion for the critical level stated in Proposition \ref{p:cont} below. These intermediate results are the focus of the subsequent sections: in Section \ref{s:prelim} we collect preliminary results on shot noise and Gaussian fields, in Section \ref{s:conv} we prove the strong invariance principle in Theorem~\ref{t:sip}, and finally in Section \ref{s:cont} we prove the continuity criteria in Proposition~\ref{p:cont}.

\subsection{Acknowledgements}
The second author was supported by the Australian Research Council (ARC) Discovery Early Career Researcher Award DE200101467. The authors thank Chinmoy Bhattacharjee, Michael Goldman, Martin Huesmann, Manjunath Krishnapur and Felix Otto for helpful discussions on strong invariance principles for Poisson random measures, and Giovanni Peccati for suggesting to extend the strong invariance principle to derivatives of the field.

\medskip
\section{The asymptotic expansion of the critical level}
\label{s:proof}

In this section we prove our main result (Theorem~\ref{t:main}) subject to intermediate statements which are established in the subsequent sections. Throughout this section we restrict our attention to the planar case $d=2$.

\smallskip The proof consists of using the convergence of shot noise fields to a limiting Gaussian field, described in Section \ref{s:gaus} above, to `pull back' known results on level set percolation of the Gaussian limit (see Theorems \ref{t:qualzero} and \ref{t:quantzero} below). In order to achieve this we use a quantitative continuity result (see Proposition \ref{p:cont} below) that relates the critical level of shot noise fields with that of the Gaussian limit.

\subsection{Level set percolation of Gaussian fields}
In this section we recall the relevant results on level set percolation of planar Gaussian fields $f$. By planar duality, and the fact that $f$ and $-f$ are equal in law, it is natural to expect that $\ell_c(f) = 0$ for centred stationary planar Gaussian fields. Indeed this, and more precise descriptions of the phase transition at $\ell = 0$, have recently been established for a wide class of Gaussian fields (see \cite{bg17, rv20, mv20, mrv20} and references therein).

\smallskip
Let us introduce first the concept of a \textit{sharp threshold} for rectangle crossing events. For a level $\ell \in \R$ and $a,b > 0$, let $\cross_\ell[a, b]$ denote the event that $\{f \le \ell \} \cap ([0,a] \times [0,b])$ contains a path that crosses the rectangle $[0,a] \times [0,b]$ from left to right (i.e.\ intersects both $\{0\} \times [0, b]$ and $\{a\} \times [0, b]$). We say that $f$ has a \textit{sharp threshold (at $\ell = 0$)} if, for every $w > 0$ and $\rho > 0$,
\begin{equation}
\label{e:ncw1}
\lim_{R \to \infty} \P[ \textrm{Cross}_{-w}[\rho R, R] ] =0  \quad \text{and} \quad \lim_{R \to \infty} \P[ \textrm{Cross}_{w}[\rho R, R] ] = 1 .
\end{equation}
If $f$ has symmetry under rotation by $\pi/2$, then since $f$ and $-f$ are equal in law each of the statements in \eqref{e:ncw1} implies the other (see the discussion at the beginning of Section \ref{sec:continuity}).

\begin{theorem}[{\cite[Theorem 1.3]{mrv20}}]
\label{t:qualzero}
Let $f$ be a $C^3$-smooth centred stationary isotropic planar Gaussian field with covariance $K$ satisfying:
\begin{itemize}
\item $K(x) (\log \log |x|)^3 \to 0$ as $|x| \to \infty$;
\item The vector $( f(0), f(x), \nabla f(0), \nabla f(x) )$ is non-degenerate for each $x \in \R^d \setminus \{  0 \}$.
\end{itemize}
Then $\ell_c(f) = 0$ and $f$ has a sharp threshold in the sense of \eqref{e:ncw1}.
\end{theorem}
\begin{remark}
The fact that $\ell_c(f) = 0$ and the existence of a sharp threshold in the sense of~\eqref{e:ncw1} are closely related, but it is not known that they are equivalent in general. In \cite{mrv20} a quantitative version of \eqref{e:ncw1} was used to deduce that $\ell_c(f) = 0$, and for our purposes it turns out that $\ell_c(f) = 0$ alone is insufficient (see however Remark \ref{r:nosharp}).
\end{remark}

While the conclusions of Theorem \ref{t:qualzero} will be sufficient to prove \eqref{e:main1} in Theorem \ref{t:main}, for the stronger result \eqref{e:main2} we will need a quantitative improvement on \eqref{e:ncw1} proven in \cite{mv20} under stronger conditions (although this improvement is believed to be true more generally). We say that $f$ has a \textit{near-critical window of size $w$ (at $\ell = 0$)} if there exists a positive function $w = w(R) > 0$ such that, for every $\rho > 0$,
\begin{equation}
\label{e:ncw2}
\lim_{R \to \infty} \P[ \textrm{Cross}_{-w(R)}[\rho R, R] ] = 0  \quad \text{and} \quad \lim_{R \to \infty} \P[ \textrm{Cross}_{w(R)}[\rho R,  R] ] = 1 .
\end{equation}
This property is monotonic in $ w$ in the sense that if it holds for some $w$ it holds for any $w' \ge w$. As for \eqref{e:ncw1}, if $f$ is Gaussian and has symmetry under rotation by $\pi/2$ then if one of the statement holds for all $ \rho >0$, so does the other. Note that if $f$ has a sharp threshold then $f$ has a near-critical window of size $w$ for arbitrary small constant $w(R) = w' > 0$. We say that $f$ has a $\textit{polynomial near-critical window}$ if it has a near-critical window of size $w(R) = R^{-c}$ for some $c > 0$.

\begin{theorem}[{\cite[Theorem 1.15]{mv20}}]
\label{t:quantzero}
Let $g : \mathbb{R}^2 \to \mathbb{R}$ be a $C^3$-smooth positive isotropic function, and assume there exists a $\beta > 2$ such that, for every multi-index $\alpha$ such that $|\alpha| \le 3$,
\[ |\partial^\alpha g(x)| \le (1+ |x|)^{-\beta} , \quad x \in \R^2. \]
Let $f$ be the $C^2$-smooth centred stationary isotropic planar Gaussian field whose covariance is $K(x) = \int g(-y) g(x-y) \, dy $. Then $\ell_c(f) = 0$ and $f$ has a polynomial near-critical window.
\end{theorem}

\begin{remark}
By analogy with planar Bernoulli percolation, it is natural to conjecture that the true near-critical window is of polynomial order with exponent $1/\nu = 3/4$ (where $\nu = 4/3$ is the \textit{correlation length} exponent for Bernoulli percolation, see \cite[Chapters 9 \& 10]{gr99}), in the sense that \eqref{e:ncw2} holds for $w(R) R^{3/4} \to \infty$ but does not hold if $w(R) R^{3/4} \to 0$. In \cite{mv20} it was shown that the exponent satisfies $1/\nu \le 1$ if it exists. On the other hand, if one relaxes the decay assumption on $g$ but instead requires that $g$ be regularly varying at infinity with index $-\beta \in (-\infty,-1)$ then one might expect that in general the critical window is of polynomial order with exponent $1/\nu = \min\{\beta-1, 3/4\}$.
 \end{remark}

\subsection{Continuity of the critical level}
\label{sec:continuity}
In order to take advantage of the convergence of the shot noise field to a Gaussian limit (Theorem \ref{t:sip}), we need a quantitative continuity result that relates critical levels through the limit. We state this continuity result in a general form that applies to arbitrary sequences of stationary planar random fields that converge to a limit.

\smallskip
Let us first introduce a mild regularity assumption on the fields. Recall the crossing event $\textrm{Cross}_\ell[a, b]$, and let $\textrm{Cross}^{tb}_\ell[a, b]$ denote the `top-bottom' version of this event, that is, the event that $\{f \le \ell \} \cap ([0,a] \times [0,b])$ contains a path that intersects both $[0,a] \times \{b\}$ and $[0,a] \times \{0\}$. We say that a planar random field $f$ is \textit{regular} if, for each $\ell \in \R$ and $a,b \ge 0$, all translations and rotations of the events $\{ \pm f \in \textrm{Cross}_\ell[a, b]\}$ are measurable, and the events
\[ \{f \in  \textrm{Cross}_\ell[a, b] \} \quad \text{and} \quad \{ - f \in  \textrm{Cross}^{tb}_{-\ell}[a, b] \} \]
partition the probability space up to a null set. In other words, for any rectangle $R$ and $\ell \in \R$ almost surely exactly one of the following holds: (i) $R \cap \{f \le \ell\}$ contains a path from left to right, or (ii) $R \cap \{f  \ge \ell \}$ contains a path from top to bottom.
\smallskip
Under Assumptions \ref{a:3} and \ref{a:2} the shot noise fields $F_\lambda$ and the Gaussian field $f$ are regular. This is a consequence of their level lines being $C^1$-smooth curves which almost surely do not have tangential intersections with fixed line-segments (see \cite[Lemma 3.1]{lm19} and \cite[Lemma A.9]{rv19} respectively for details).

\smallskip
Let us also introduce the concept of \textit{sprinkled decoupling}, which is the main property we use to guarantee continuity of the critical level (see also \cite{pr15,mv20} for other appearances of this property). Recall the crossing events $\textrm{Cross}_\ell[a, b]$, which are increasing with respect to the field. We say that a planar regular random field $f$ satisfies the \textit{sprinkled decoupling property} with \textit{sprinkling function} $h(R) \ge 0$ and \textit{error function} $e(R) \ge 0$ if, for any $R \ge 1$ and $\ell \in \R$, and any events $A, B$ which are translations and rotations of $\{  f \in \textrm{Cross}_\ell[R, 3R] \}$ such that the distance between the associated rectangles on which $A$ and $B$ are defined is at least~$R$, it holds that
\begin{equation}
\label{e:sd}
\mathbb{P}[ f  \in  A \cap B ] \le \mathbb{P}[ f + h(R) \in A] \mathbb{P}[f + h(R) \in B]  + e(R) .
\end{equation}
This property is monotonic in $(h,e)$, in the sense that if it holds for $(h,e)$ it also holds for any $(h', e')$ such that $h' \ge h$ and $e' \ge e$.

\smallskip If the sprinkled decoupling property holds for $h = 0$ and $e(R) \to 0$, and if the field $f$ is positively-associated so that $\mathbb{P}[ f  \in  A \cap B ] \ge \mathbb{P}[ f  \in A] \mathbb{P}[f \in B]$, then we recover the usual notion of \textit{asymptotic independence}
\[ | \mathbb{P}[ f  \in  A \cap B ] - \mathbb{P}[ f  \in A] \mathbb{P}[f \in B] | \le e(R) \to 0  .\]
We shall rather be interested in the case that $(h,e)$ can be chosen to decay polynomially, and we say that $f$ satisfies the \textit{polynomial sprinkled decoupling property} if there exist constants $c_h, c_{e, 1}, c_{e,2} > 0$ such that $f$ satisfies the sprinkled decoupling property with $h(R) = R^{-c_h}$ and $e(R) = c_{e,1}R^{-c_{e,2}}$. We say that a collection $(f_n)$ of random fields satisfies the \textit{uniform polynomial sprinkled decoupling property} if each $f_n$ satisfies the polynomial sprinkled decoupling property for constants $(c_h, c_{e,1}, c_{e,2})$ chosen uniformly over the collection.

\smallskip
The next proposition verifies that, under the conditions in Assumption \ref{a:2}, for any $\eta > 1/2$ the rescaled shot noise fields $( \lambda^{-\eta}  f_\lambda)_{\lambda \ge 1}$ satisfy this property (see Section \ref{sec:prf-poly-sprink} for the proof):

\begin{proposition}
\label{p:sd}
Assume that the kernel $g$ satisfies Assumption \ref{a:2} for some $\beta > 2$ and let $\eta > 1/2$. Then the collection of random fields $( \lambda^{-\eta}  f_\lambda)_{\lambda \ge 1}$ satisfies the uniform polynomial sprinkled decoupling property.
\end{proposition}

\begin{remark}
As shown in the proof, the implicit constant $c_h$ in Proposition \ref{p:sd} can be chosen to be any $c_h < \beta - 2$, although we do not use this fact.
\end{remark}
\begin{remark}
While it may be the case that the fields $(f_\lambda)_{\lambda \ge 1}$ also satisfy the uniform polynomial sprinkled decoupling property, our proof only establishes it for the rescaled fields $(\lambda^{- \eta}  f_\lambda)_{\lambda \ge 1}$ if $\eta > 1/2$; see Remark \ref{r:upsd}.
\end{remark}

We can now state our quantitative continuity result for sequences of stationary planar random fields which converge to a limit  (see Section \ref{s:cont} for the proof):

 \begin{proposition}
 \label{p:cont}
Let $(f_n)_{n \ge 1}$ and $f$ be stationary planar regular random fields. Suppose there exists positive functions $s(n), u(n), R(n)$ and $w(R)$ such that:
\begin{itemize}
\item The collection of rescaled fields $(s(n) f_n)_{n \ge 1}$ satisfies the uniform polynomial sprinkled decoupling property with positive constants $(c_h, c_{e,1}, c_{e,2})$.
\item The random field $f$ has a near-critical window (at $\ell = 0$) of size $w(R)$.
\item $R(n) \to \infty$ as $n \to \infty$, and for every $n \ge 1$ there exists a coupling of $f_n$ and $f$ satisfying, as $n \to \infty$,
\begin{equation}
\label{e:un}
 \P[  \| f_n - f \|_{\infty, B(2R(n))}  \ge  u(n) ] \to  0.
 \end{equation}
\end{itemize}
 Then there exists a $c > 0$, depending only on $c_h$, such that
 \begin{equation}
 \label{e:cont}
 |\ell_c(f_n)| \le  3 \max\{u(n), w(R(n)), c \, s(n)^{-1} R(n)^{-{c_h} } \}
 \end{equation}
 for sufficiently large $n$.
 \end{proposition}

\begin{remark}
It is not essential that the uniform sprinkled decoupling satisfied by the fields $f_n$ is `polynomial' for a continuity result such as Proposition \ref{p:cont} (with suitable adjustments) to hold; what is required is that $h$ and $e$ are summable over a dyadic sequence, i.e.\
\[ \sum_{n \ge 1} h(2^n) < \infty \qquad \text{and} \qquad \sum_{n \ge 1} e(2^n) < \infty , \]
so for instance it would be enough if $h$ and $e$ decay at order $(\log R)^{-c}$ for some $c > 1$; see the proof of Proposition \ref{p:fsc} for details. We use polynomial sprinkled decoupling for convenience and since it already holds under our assumptions.
\end{remark}

\subsection{Proof of Theorem \ref{t:main}}
Fix $\eta > 1/2$, and recall that by Proposition \ref{p:sd} the collection of random fields $(\lambda^{-\eta} f_\lambda)_{\lambda \ge 1}$ satisfies the uniform polynomial sprinkled decoupling property with some constant $c_h > 0$. Moreover, one can check that Assumption \ref{a:1} guarantees that the Gaussian field $f$ satisfies the conditions of Theorem \ref{t:qualzero}, and hence $f$ has a near-critical window of size $w$ for arbitrary small constant $w(R) = w' > 0$. Define $R(\lambda) = \lambda^{c_3}$ and $u(\lambda) =  c_4  \lambda^{-1/2} (\log \lambda)^{3/2}$ for constants $c_3 > \eta/c_h$ and $c_4 > (c_{2,\beta} + 2 c_3)/c_2$, where $c_{2,\beta} > 0$ is defined in \eqref{e:sigma} and $c_2 > 0$ is the constant appearing in Theorem \ref{t:sip}. Then by the strong invariance principle in Theorem \ref{t:sip} (applied to $t  = c_4 \log \lambda)$ there exists a $c_1 > 0$ (independent of the choice of $c_3, c_4$) such that, for every $\lambda \ge 1$ sufficiently large there exists a coupling of $f_\lambda$ and $f$ satisfying
\begin{equation}
\label{e:pmain1}
\P \Big[  \| f_\lambda - f \|_{\infty, 2B(R(\lambda) )}  \ge u(\lambda)  \Big]  \le c_1 \lambda^{c_{2, \beta}} R(\lambda)^2 e^{-c_2 t} = c_1 \lambda^{c_{2 , \beta} + 2 c_3 } e^{-c_2 c_4 \log \lambda} .
\end{equation}
By the choice of $c_4 > 0$, the right-hand side of \eqref{e:pmain1} tends to zero as $\lambda \to \infty$. Hence the assumptions in Proposition~\ref{p:cont} are satisfied, and we deduce that
\[ |\ell_c(f_\lambda)| \le c_5 \max\{   \lambda^{-1/2} (\log \lambda)^{3/2}  ,  w' ,  \lambda^{\eta - c_h c_3  }\}   \]
for some $c_5 > 0$. Recalling that $c_3 > \eta/c_h$, by taking $\lambda \to \infty$ and then $w' \to 0$ we see that $\ell_c(f_\lambda) \to 0$, which is \eqref{e:main1}.

\smallskip
For the second statement \eqref{e:main2}, observe that Assumption \ref{a:1} and the extra condition that $g \ge 0$ implies that $f$ satisfies the assumptions of Theorem \ref{t:quantzero}, and hence $f$ has a near-critical window of size $w(R) = R^{-c_6}$ for some $c_6 > 0$. Redefine $R = R(\lambda) = \lambda^{c_7}$ and $u(\lambda) =  c_8  \lambda^{-1/2} (\log \lambda)^{3/2}$ for constants $c_7 > \max\{(1/2+\eta)/c_h, 1/(2c_6)\}$ and $c_8 > (c_{2,\beta} + 2 c_7)/c_2$. Then by the strong invariance principle in Theorem~\ref{t:sip}, for every $\lambda \ge 1$ sufficiently large there exists a coupling of $f_\lambda$ and $f$ satisfying
\begin{equation}
\label{e:pmain2}
\P \Big[  \| f_\lambda - f \|_{\infty, 2 B(R(\lambda))}  \ge u(\lambda)  \Big]  \le c_1 \lambda^{c_{2, \beta}} R(\lambda)^2 e^{-c_2 t} = c_1 \lambda^{c_{2, \beta} + 2 c_7 } e^{-c_2 c_8 \log \lambda} .
\end{equation}
Again by the choice of $c_4 > 0$ the right-hand side of \eqref{e:pmain2} tends to zero as $\lambda \to \infty$. Hence the assumptions in Proposition~\ref{p:cont} are satisfied, and we deduce that
\[| \ell_c(f_\lambda) | \le  c_9 \max\{  \lambda^{-1/2} (\log \lambda)^{3/2} , w( \lambda^{c_7})  ,  \lambda^{1/2  - c_h c_7 }    \} \]
for some $c_9 > 0$. Since $w(\lambda^{c_7}) = \lambda^{-c_6 c_7} = o( \lambda^{-1/2} ) $ and $\lambda^{\eta - c_h c_7 }  = o( \lambda^{-1/2} ) $  (recall that $c_7 > \max\{(1/2+\eta)/c_h, 1/(2 c_6)\}$), this implies that
\[  |\ell_c(f_\lambda)| = O( \lambda^{-1/2} (\log \lambda)^{3/2} ) . \]

\begin{remark}
\label{r:nosharp}
To deduce \eqref{e:main1} from Proposition \ref{p:cont} we used the existence of a sharp threshold in \eqref{e:ncw1}. On the other hand, it is perhaps more natural to work with the assumption $\ell_c(f) = 0$ rather than~\eqref{e:ncw1}. As we now explain, if we additionally assume that $f$ is positively-associated (equivalent to assuming $K(x) \ge 0$ \cite{pit82}, and hence implied by $g(x) \ge 0$) then in fact $\ell_c(f) = 0$ implies \eqref{e:ncw1}. Indeed let us consider the contrapositive statement, i.e., we choose an $\ell > 0$ such that
\begin{equation}
\label{e:nosharp}
 \liminf_{R \to \infty} \P[ \textrm{Cross}_{-\ell}[R,2R] ] >  0 ,
 \end{equation}
and aim to show that this implies $\ell_c(f) \ge \ell/2 > 0$. By applying Russo-Seymour-Welsh theory \cite{att18,kt20}, available since we assume stationarity, isotropy, and positive associations, we can deduce from \eqref{e:nosharp} that
\[\liminf_{R \to \infty} \P(\textrm{Ann}_{-\ell}[R,2R]) > 0 , \]
where $\textrm{Ann}_{-\ell}[R, 2R]$ is the event that $\{f \le -\ell\} \cap ( [-2R, 2R]^2 \setminus [-R, R]^2 )$ contains a circuit. Then an ergodic argument (see the `box lemma'  in \cite{gkr88}, and recall that we assume $K(x) \to 0$ as $|x| \to \infty$ which implies that $f$ is ergodic) shows that
\[   \P[ \text{every bounded set $A \subseteq \R^2$ is surrounded by a circuit in }  \{f \le -\ell\}  ] = 1 . \]
This implies that $\{f > -\ell\}$ almost surely does not have an unbounded component, and by symmetry in law of $f$ and $-f$, also that $\{f < \ell\}$ does not have an unbounded component. Hence $\ell_c(f) \ge \ell > 0$, as we aimed to show.
\end{remark}

\medskip
\section{Properties of shot noise and Gaussian fields}
\label{s:prelim}

In this section we collect some preliminary estimates on shot noise and Gaussian fields on~$\R^d$, in particular (i) we show how to couple these fields to approximations that have finite-range dependence, and (ii) we give bounds on their $C^1$-norm. We then use the finite-range approximation to establish the polynomial sprinkled decoupling property in Proposition \ref{p:sd}.

\smallskip
Throughout this section we assume for simplicity that the kernel $g$ satisfies~Assumption \ref{a:2} for some $\beta > d$, although some of the results hold under weaker assumptions. {  All of the constants defined in this section depend only on $\beta$ and $d$.}

\subsection{Convolution representations of shot noise and Gaussian fields}
\label{ss:conrep}
We begin by recalling a convenient representation of shot noise and Gaussian fields. Recall that $f_\lambda$ denotes the rescaled and centred shot noise field
\[  f_\lambda(x) = \frac{F_\lambda(x) - \E[ F_\lambda(0)] }{\lambda^{1/2}}  =  \frac{ \sum_{i \in \mathcal{P}_\lambda} g(x-i)  - \lambda \smallint g }{\lambda^{1/2}}  , \]
and that $f$ is the stationary centred Gaussian field $f$ with covariance
\[ K(x) = \mathbb{E}[f(0) f(x)] = \int g(-y) g(x-y) \, dy  . \]
Clearly we may represent $f_\lambda$ as
\[  f_\lambda(x) = (g \star \widetilde{\mathcal{P}_\lambda})(x) = \widetilde{\mathcal{P}_\lambda}( g(x-\cdot ) )  , \]
 where $\star$ denotes convolution, $\widetilde{\mathcal{P}_\lambda}$ is the normalised compensated Poisson measure
 \[  \widetilde{\mathcal{P}_\lambda} = \frac{\sum_{i \in \mathcal{P}_\lambda} \delta_{i}  - \lambda dx  }{\lambda^{1/2}} \ , \qquad (\widetilde{\mathcal{P}_\lambda}(g))_{g \in L^1(\R^d)} =  \Big(\frac{\sum_{i \in \mathcal{P}_\lambda} g(i)  -  \lambda \int g   }{\lambda^{1/2}} \Big)_{g \in L^1(\R^d)}   ,     \]
 $\delta_i$ denotes a Dirac mass at $i \in \R^d$, and $dx$ is the Lebesgue measure. Similarly, recall the `stationary moving average' representation of the Gaussian field $f$

\begin{align}
\label{eq:gauss-convol}
f(x) = (g \star \widetilde{\mathcal{W}})(x) =  \widetilde{\mathcal{W}}( g(x-\cdot) ) ,  
\end{align}
where $\widetilde{\mathcal{W}}$ denotes the white noise on $\R^d$, i.e.\ $(\widetilde{\mathcal{W}}(g))_{g \in L^2(\R^d)}$ is the centred Gaussian process such that $\E[ \widetilde{\mathcal{W}}(g_1) \widetilde{\mathcal{W}}(g_2) ] = \int g_1 g_2$. To justify this representation, observe that $g \star \widetilde{\mathcal{W}}$ is a stationary centred Gaussian field with covariance kernel
\[ K(x) = \E[ f(0) f(x) ] = \E \big[ \widetilde{\mathcal{W}}( g(0 - \cdot) ) \widetilde{\mathcal{W}}(g (x - \cdot)) \big] = \int g(-y) g(x-y) dy  , \]
and so has the same law as the Gaussian field $f$ defined in Section \ref{s:gaus}. 

\subsection{Finite-range approximation}
The representations $f_\lambda = g \star \widetilde{\mathcal{P}_\lambda}$ and $f = g \star \widetilde{\mathcal{W}}$ allow us to construct approximations $f^r_\lambda$ and $f^r$ of the fields $f_\lambda$ and $f$ respectively, that are (i) coupled on the same probability space, and (ii) $r$-range dependent, meaning that events supported on sets separated by a distance $r$ are independent. In fact we may construct these approximations simultaneously for all $r > 0$.

\smallskip
For this we introduce a smooth cut-off function $\chi: \R^d \to [0, 1]$ such that $\chi(x) = 1$ for  $x \in B(1/4)$ and $\chi(x) = 0$ for $x \notin B(1/2)$; the choice of $\chi$ is arbitrary but the unspecified constants in our results depend on it (including Theorem \ref{t:main}). Then for $r > 0$, define the truncated kernel $g^r(x) = g(x) \chi(x/r)$, and the fields
\[ f_\lambda^r(x) = (g^r \star \widetilde{\mathcal{P}_\lambda})(x)  =  \widetilde{\mathcal{P}_\lambda}( g^r(x-\cdot) )  = \frac{ \sum_{i \in \mathcal{P}_\lambda} g^r(x-i)  - \lambda \smallint g^r }{\lambda^{1/2}}     \]
and
\[ f^r(x) = (g^r \star \widetilde{\mathcal{W}})(x) =  \widetilde{\mathcal{W}}( g^r(x-\cdot) )  .\]
Since $g^r$ is supported on $B(r/2)$, the fields $f_\lambda^r$ and $f^r$ are $r$-range dependent, and they are naturally coupled to $f_\lambda$ and $f$, respectively, via $\widetilde{\mathcal{P}_\lambda}$ and $\widetilde{\mathcal{W}}$.

\smallskip
In the following two propositions we control the quality of these approximations. Recall that $\beta > d$ is the constant in Assumption \ref{a:2}. 

\begin{proposition}
\label{p:snsup}
There exist $c_1, c_2 > 0$ such that, for every $\lambda ,  r , u \ge 1$
\[ \P \big[ \|f_\lambda - f^r_\lambda \|_{[0,1]^d, \infty} \ge u \lambda^{1/2} r^{d-\beta}  \big] \le c_1 \lambda e^{- c_2 \sqrt{ u }} .\]
\end{proposition}

\begin{proof}
By the definition of $f_\lambda^r$ and $f_\lambda$ we have
\[ (f_\lambda^r - f_\lambda)(x) = ((g - g^r) \star \widetilde{\mathcal{P}_\lambda})(x) =  \widetilde{\mathcal{P}_\lambda}( (g-g^r)(x - \cdot) ) =  \frac{ \hat{F}^r_\lambda(x)  - \lambda \smallint (g^r - g)  }{ \lambda^{1/2} }   , \]
where $\hat{F}^r_\lambda(x) = \sum_{i \in \mathcal{P}_\lambda}  ( g-g^r)(x-i)$ is a shot noise field with kernel $g - g^r$. By the decay of $g$ in Assumption~\ref{a:2}, there is a $c_3 > 0$ such that, for $ r \ge 1$,
\[ |  \smallint (g^r - g) | \le c_3  r^{d-\beta}    , \]
and hence
\begin{equation}
\label{e:snsup}
  \P \big[ \|f_\lambda - f^r_\lambda \|_{[0,1]^d, \infty} \ge u \lambda^{1/2} r^{d-\beta}  \big] \le    \P \big[  \| \hat{F}^r_\lambda \|_{[0,1]^d, \infty} \ge (u  - c_3 )  \lambda r^{d-\beta} \big] .
  \end{equation}
To control this latter quantity we apply the concentration bounds for shot noise fields in \cite[Proposition A.2]{lm19}. Consider the rescaled field $ \hat{G}(\cdot) =  \hat{F}^r_\lambda(\cdot / \lambda^{1/d} )$ which is a unit-intensity shot noise field with kernel $h(\cdot) = (g - g^r)(\cdot / \lambda^{1/d})$, and define the auxiliary function on $\R^d$
\[  j(x) = \sup_{y \in [-1/2,1/2]^d }   |h(x+y)|  . \]
Applying \cite[Proposition A.2, (A.1)]{lm19} to  $\hat{G}$ (where $j$ corresponds to $\hat{h}$ using the notation in this proposition, and we may take $\beta = \infty$ in this proposition since we work with degenerate marks), there exist $c_4, c_5 > 0$ depending only on $d$ such that, for all $v \ge 1$,
\begin{align*}
   &   \P \Big[  \| \hat{F}^r_\lambda \|_{[0, 1]^d, \infty}  \ge  2  v \Big( \| j \|_{L^1} + \sqrt{2v} \| j \|_{L^2}  + \frac{v}{3}  \| j \|_\infty \Big)  \Big] \\
    & \qquad  =  \P \Big[  \| \hat{G} \|_{[0, \lambda^{1/d}]^d, \infty}  \ge  2  v \Big( \| j \|_{L^1} + \sqrt{2v} \| j \|_{L^2}  + \frac{v}{3}  \| j \|_\infty \Big)  \Big]   \\
   & \qquad \le   c_4 \lambda e^{-c_5 v} .
   \end{align*}
By the decay of $g$ in Assumption \ref{a:2}, there is a $c_6 > 0$ such that, for $r \ge 1$,
\[ \max \big\{ \| j \|_{L^1}  ,   \| j  \|_{L^2}  ,   \| j \|_\infty   \big\}  \le c_6 \lambda r^{d-\beta} , \]
so the bound reduces to
\[    \P \Big[  \| \hat{F}^r_\lambda \|_{[0, 1]^d, \infty}  \ge  c_7  v^2 \lambda  r^{d-\beta}   \Big]   \le   c_4 \lambda e^{-c_5  v}  \]
for some $c_7 > 0$. Assuming $u \ge c_3 + c_7 $ we may apply this bound to $v^2 = (u - c_3) / c_7  \ge 1$, and combining with \eqref{e:snsup} we deduce that
\[   \P \big[ \|f_\lambda - f^r_\lambda \|_{[0,1]^d, \infty} \ge u \lambda^{1/2} r^{d-\beta}  \big] \le  c_4 \lambda e^{-c_8\sqrt{u - c_3 } }  \]
for some $c_8 > 0$. This proves the result for sufficiently large $u$, and hence for all $u \ge 1$ after adjusting constants.
\end{proof}

\begin{proposition}
\label{p:gfsup}
There exist $c_1, c_2 > 0$ such that, for every $r , u \ge 1$
\[ \P \big[ \|f - f^r \|_{[0,1]^d, \infty} \ge u r^{d/2-\beta}  \big] \le c_1 e^{- c_2 u^2} .\]
\end{proposition}
\begin{proof}
The case $d=2$ is proven in \cite[Proposition 3.11]{mv20}, and the proof extends immediately to arbitrary dimension $d \ge 1$ (with obvious adjustments).
\end{proof}

\subsection{Derivative bounds}

For a domain $D \subset \R^d$, let $\| h \|_{C^1(D)} = \sup_{|\alpha| \le 1}  \| \partial^\alpha h \|_{ \infty, D } $. The following results bound the $C^1$-norm of the fields $f_\lambda$ and $f$:

\begin{proposition}
\label{p:sndb}
There exist $c_1, c_2 > 0$ such that, for $\lambda, u \ge 1$,
\[ \P \big[ \|f_\lambda \|_{C^1([0,1]^d)} \ge u \lambda^{1/2}  \big] \le c_1 \lambda e^{- c_2 \sqrt{u} } .\] 
\end{proposition}

\begin{proof}
For each multi-index $\alpha$ such that $|\alpha| \le 1$, $\lambda^{1/2} (\partial^\alpha f_\lambda)$ is distributed as a shot noise field with kernel $\partial^\alpha g$ (since we may interchange derivative and summation by dominated convergence). Hence we can control the $C^1$-norm of $f_\lambda$ using the concentration bounds for shot noise fields in \cite[Proposition A.2]{lm19} as in the proof of Proposition \ref{p:snsup}. 

For $|\alpha| \le 1$, consider the unit-intensity shot noise field $\hat{G}_\alpha(\cdot) =  \lambda^{1/2} (\partial^\alpha f_\lambda)(\cdot / \lambda^{1/d})$ and define the auxiliary function on $\R^d$
\[  j_\alpha(x) = \sup_{y \in [-1/2,1/2]^d }   |\partial^\alpha g ((x+y)/\lambda^{1/d})|  . \]
According to \cite[Proposition A.2, (A.1)]{lm19} (again we may take $\beta = \infty$ in this proposition since we work with degenerate marks), there exist $c_3, c_4 > 0$ depending only on $d$ such that, for all $u \ge 1$,
\begin{align*}
&   \P \Big[  \lambda^{1/2} \|  \partial^\alpha f_\lambda \|_{[0, 1]^d, \infty}  \ge  2 u \Big( \| j_\alpha \|_{L^1} + \sqrt{2u} \| j_\alpha \|_{L^2}  + \frac{u}{3}  \| j_\alpha \|_\infty \Big)  \Big]   \\
& \qquad  =  \P \Big[  \| \hat{G}_\alpha \|_{[0, \lambda^{1/d} ]^d, \infty}  \ge  2 u \Big( \| j_\alpha \|_{L^1} + \sqrt{2u} \| j_\alpha \|_{L^2}  + \frac{u}{3}  \| j_\alpha \|_\infty \Big)  \Big]  \\
& \qquad \le   c_3 \lambda e^{-c_4 u} .
\end{align*}
By the decay of $g$ and its derivatives in Assumption \ref{a:2},
\[ \max \big\{ \| j_\alpha \|_{L^1}  ,   \| j_\alpha  \|_{L^2}  ,   \| j_\alpha \|_\infty   \big\} \le c_5 \lambda , \]
for some $c_5 > 0$, so the bound reduces to
\[    \P \Big[  \lambda^{1/2} \|   \partial^\alpha f_\lambda \|_{[0, 1]^d, \infty}  \ge  c_6 \lambda u^2    \Big]   \le   c_3 \lambda e^{-c_4 u} \]
for some $c_6 > 0$. We obtain the desired result by summing up the bound for all $|\alpha| \le 1$ and adjusting constants.
\end{proof}

\begin{proposition}
\label{p:gfdb}
There exist $c_1, c_2 > 0$ such that, for $u \ge 1$,
\[ \P \big[ \|f \|_{C^1([0,1]^d)} \ge u  \big] \le c_1 e^{- c_2 u^2} .\]
\end{proposition}
\begin{proof}
{  By Kolmogorov's theorem (see \cite[Appendix~A]{ns16}), $\E[ \|f \|_{C^1([0,1]^d)} ]$ is bounded by a constant depending only on $\beta$ and $d$,} and the result then follows from the BTIS theorem \cite[Theorem 2.9]{aw09}.
\end{proof}

\subsection{Sprinkled decoupling for shot noise fields}
\label{sec:prf-poly-sprink}
In this section we prove the sprinkled decoupling property for shot noise fields stated in Proposition \ref{p:sd}. We shall make use of the following observation: if $A$ is an increasing event supported on $D \subset \R^d$, and $f,g$ are fields that are coupled on the same probability space, then for any $\varepsilon > 0$,
\begin{align}
\nonumber \P[ f \in A ] & = \P[ f \in A , \|f-g\|_{D, \infty} \le \varepsilon  ]  + \P[ f \in A, \|f-g\|_{D, \infty} > \varepsilon  ]   \\
\label{e:c} & \le \P[ g + \varepsilon \in A ] + \P[\|f-g\|_{D, \infty} \ge \varepsilon  ]   .
\end{align}

\begin{proof}[Proof of Proposition \ref{p:sd}]
Fix $\eta > 1/2$ and $\delta \in (0, \beta-d)$, and let $\lambda \ge 1$, $R \ge 1$, $\ell \in \R$, and $A, B$ be events which are translations or rotations of $\{ f \in \textrm{Cross}_\ell[R, 3R] \}$ such that the distance between their associated rectangles, which we denote by $R_A$ and $R_B$ respectively, is at least $R$. Since $A$ and $B$ are increasing, by \eqref{e:c} we have
\begin{align*}
 & \mathbb{P}[ \lambda^{-\eta} f_\lambda \in  A \cap B ]  \\
 &   \le \mathbb{P} \big[ \lambda^{-\eta} f^R_\lambda  + R^{d-\beta + \delta}  \in  A \cap B    \big]  +   \P \big[ \| f_\lambda - f^R_\lambda\|_{\infty, R_A \cup R_B } \ge \lambda^\eta  R^{d-\beta + \delta} \big]   \\
 &   \le \mathbb{P} \big[ \lambda^{-\eta} f^R_\lambda  + R^{d-\beta + \delta} \in  A \big] \mathbb{P} \big[ \lambda^{-\eta} f^R_\lambda + R^{d-\beta + \delta}  \in  B \big]   +   \P[ \| f_\lambda - f^R_\lambda\|_{\infty, R_A \cup R_B } \ge \lambda^\eta  R^{d-\beta+\delta} ]   ,
 \end{align*}
 where in the last step we used that $R_A$ and $R_B$ are separated by distance $\ge R$ and that the field $f_\lambda^R$ is $R$-range dependent. By Proposition \ref{p:snsup} (applied to $u = R^\delta  \lambda^{\eta-1/2}$) and the union bound we have
 \[  \P[ \| f_\lambda - f^R_\lambda\|_{\infty, R_A \cup R_B } \ge \lambda^\eta  R^{d-\beta + \delta} ]   \le c_1 R^d \lambda e^{-c_2 \sqrt{ R^\delta \lambda^{\eta-1/2} } }   \]
 for some $c_1, c_2 > 0$. Hence the field $\lambda^{-\eta} f_\lambda$ satisfies sprinkled decoupling with polynomial sprinkling function $h(R) = R^{d-\beta+\delta}$ and error function $e(R) =  c_1 R^d \lambda e^{-c_2 \sqrt{R^\delta \lambda^{\eta-1/2}  } } \le c_3 R^{-c_4} $ for some $c_3, c_4 > 0$ which are uniform in $\lambda \ge 1$.
\end{proof}

\begin{remark}
\label{r:upsd}
The reason that we restrict $\eta > 1/2$ in Proposition \ref{p:sd} is to ensure that the error function $e(R) =  c_1 R^d \lambda e^{-c_2 \sqrt{R^\delta \lambda^{\eta-1/2}  } }$ is uniformly decaying in $R$; indeed if $\eta \le 1/2$ then $e(R)$ is only small if 
\[R \gg \lambda^{(1/2-\eta)/\delta} (\log \lambda)^{2/ \delta} \to \infty .\]
\end{remark}

\medskip
\section{A strong invariance principle for shot noise fields}
\label{s:conv}

In this section we prove the strong invariance principle for shot noise fields in arbitrary dimension $d \ge 1$, stated in Theorem \ref{t:sip}. Recall the representations $f_\lambda(x)  = \widetilde{\mathcal{P}_\lambda}( g(x-\cdot ) )$ and $f(x)  = \widetilde{\mathcal{W}}( g(x-\cdot ) )$ from Section \ref{ss:conrep}. Our strategy is to first establish a coupling of $\widetilde{\mathcal{P}_\lambda}$ and $\widetilde{\mathcal{W}}$ such that $\widetilde{\mathcal{P}_\lambda}(g)$ and $\widetilde{\mathcal{W}}(g)$ are close with high probability for a single { compactly-supported} test function $g$ (see Proposition \ref{p:sip} below). This is achieved by applying the results of Koltchinskii \cite{kol94} on strong invariance principles for the empirical process, based on the dyadic coupling scheme of \cite{kmt75}. We then apply this to translations of the (truncated) kernel $g^r(x - \cdot)$ on a sufficiently fine mesh $x \in \varepsilon \Z^d$, and use continuity properties of the fields $f_\lambda$ and $f$ to deduce the result.

\subsection{Strong invariance principle for a single test function}
\label{s:stf}

Let us introduce the quantities which control the quality of the coupling. Let $I = [0, 1)^d$ denote the unit cube. A \textit{binary expansion} of $I$ is a sequence $\Delta = (\Delta_j)_{j \ge 0}$ of partitions of $I$ satisfying:
\begin{itemize}
\item For $j \ge 0$, $\Delta_j = (\Delta_{j,k})_{k = 0, \ldots , 2^j - 1}$ with each $\Delta_{j, k}$ a rectangle $[a_i,b_i) \times \ldots \times [a_d, b_d)$;
\item $\Delta_{0,0} = I$, and for all $j \ge 0$ and $k = 0, \ldots, 2^j - 1$,
\[ \Delta_{j, k} = \Delta_{j+1, 2k} \cup \Delta_{j+1, 2k+1}  \quad \text{and} \quad \Delta_{j+1, 2k} \cap \Delta_{j+1, 2k+1} = \emptyset ; \]
\item For $j \ge 0$ and $k = 0, \ldots, 2^j - 1$, $\textrm{Vol}(\Delta_{j, k}) = 2^{-j}$;
\item The Borel $\sigma$-algebra on $I$ is the completion with respect to the Lebesgue measure on $I$ of the $\sigma$-algebra generated by $\cup_{i \ge 0} \Delta_i$.
\end{itemize}
In other words, $\Delta$ is an iterative division of $I$ into rectangles of equal volume that generate the Borel $\sigma$-algebra. By sequential division along the $d$ coordinate directions, one can check that there exists a binary expansion of $I$ such that (i) $\textrm{diam}(\Delta_{j,k}) \le  \sqrt{d} 2^{-\lfloor j/d \rfloor} \le \sqrt{d}  2^{1- j/d }  $ for all $j \ge 0$, and (ii) $\Delta _{j,k}$ are homothetic to a finite class of rectangles, and we henceforth fix $\Delta$ to be one such expansion.

\smallskip
For a measurable function $h : I \to \R$ and a Borel subset $D \subset I$, define the $L^2$-modulus
\[ \omega^2(h ; D ) = \frac{1}{\textrm{Vol}(D)} \int_D (h-\bar h)^{2}  , \]
where $\bar h= \textrm{Vol}(D)^{-1}\int_D h. $  For $m \ge 0$, define the `level-$m$' $L^2$-modulus
\begin{equation}
\label{e:q}
 q_m(h) =  \Big(  \sum_{j = 0}^m \sum_{k=0}^{2^{j}-1} \omega^2(h; \Delta_{j,k} )  \Big)^{1/2} .
 \end{equation}
 This quantity is homogeneous in the sense that if $h'(\cdot) = a h(\cdot)$ for a constant $a > 0$ then $q_m(h') = a q_m(h)$.
 
 \smallskip
Although not necessary to state our result, for later use we observe that the asymptotic behaviour of $q_m(h)$ as $m \to \infty$ depends on the regularity of $h$ as well as the dimension $d \ge 1$. In particular, for $h$ differentiable almost everywhere, the Poincar\'e-Wirtinger inequality yields
\[ \omega^2(h ; \Delta_{j,k}) \le \frac{ c_d \text{\rm{diam}}(\Delta _{j,k})^{2} \|\nabla h\|_{L^{2}(\Delta _{j,k})}^{2} }{\textrm{Vol}(\Delta_{j,k}) }, \]
where $c_d > 0$ is a constant depending only on the dimension, and so
\begin{align}
\label{e:qbounds}
 \nonumber q_m(h)  & \le   c'_d  \Big(  \sum_{j = 0}^m   2^{(1-2/d)j} \sum_{k}\|\nabla h\|_{L^{2}(\Delta _{j,k})}^{2}     \Big)^{1/2} \\
 & \le    c'_d \|\nabla h\|_{L^{2}(I)} \times \begin{cases}   \sqrt{2} &  d = 1 , \\   \sqrt{m+1} & d = 2 ,  \\  \sqrt{2}  \times 2^{(1/2-1/d)m} & d \ge 3 , \end{cases}
 \end{align}
 where $c'_d = 2\sqrt{d c_d}$.
 
 \smallskip
Finally, for $i \in \Z^d$ define the translated cubes $C_i = i + I$, which collectively tile $\R^d$, and for a function $g:\R^d \to \R$ we define $g_i = g|_{C_i}$ to be the restriction of $g$ to the cube $C_i$, so that $g = \sum_{i \in \Z^d} g_i$. We further define $q_m(g_i)$ in the natural way (i.e.\ by identifying $C_i$ with $I$).

\begin{proposition}
\label{p:sip}
There exists $ c_1, c_2 > 0$ depending only on the dimension such that, for every $\lambda \ge 1$ there exists a coupling of $\widetilde{\mathcal{P}_\lambda}$ and $\widetilde{\mathcal{W}}$ satisfying, for every $r \ge 1$, measurable function $g$ that is supported on $B(r)$, and every $t \ge 1 $,
\[    \P \Big[  | \widetilde{\mathcal{P}_\lambda}(g) - \widetilde{\mathcal{W}}(g) | \ge t     \lambda^{-1/2}  \sum_{i \in \Z^d}  \big( q_{[ \log_2(2\lambda/t) ]^+ }(g_i)  + \|g_i\|_\infty \big)  \Big]  \le c_1  r^d \lambda e^{-c_2 t}  \]
where $q_m(h)$ is defined in \eqref{e:q} and $[ x ]^+ = \max\{0, \lfloor x \rfloor\}$.
\end{proposition}

To prove Proposition \ref{p:sip} we rely on the results of Koltchinskii \cite{kol94} on strong invariance principles for the empirical process (see also \cite{rio94} for related results), which we now recall. For $n \ge 1$, the empirical process on $I$ (equipped with the Lebesgue measure) is the sequence of random measures
\begin{equation}
\label{e:ep}
\widetilde{\mathcal{E}_n} =  \frac{ \sum_{i = 1}^n \delta_{X_i} - n dx }{\sqrt{n}}   \ , \qquad  (\widetilde{\mathcal{E}_n}(h))_{h \in L^1(I)} = \Big( \frac{ \sum_{i=1}^n h(X_i) - n \int h}{\sqrt{n}} \Big)_{ h \in L^1(I)} ,
 \end{equation}
where $(X_i)_{i \ge 1}$ is an i.i.d.\ sequence of random variables distributed uniformly on $I$. Further let $\widetilde{\mathcal{W}}_0$ denote the `Brownian bridge process' on $I$, i.e.\ $(\widetilde{\mathcal{W}}_0)_{h \in L^2(I)}$ is the centred Gaussian process such that $\E[ \widetilde{\mathcal{W}}_0(h_1) \widetilde{\mathcal{W}}_0(h_2) ] =  \int h_1 h_2  - \int h_1 \int h_2$.

\begin{proposition}[Special case of {\cite[Theorem 3.5]{kol94}}]
\label{p:kol}
There exists $ c_1, c_2 > 0$ depending only on the dimension such that, for every $n \ge 1$ there exists a coupling of $\widetilde{\mathcal{E}_n}$ and $\widetilde{\mathcal{W}}_0$ satisfying, for every measurable $h : I \to \R$ such that $\|h\|_\infty \le 1$, and every $t \ge 1$,
\[    \P \Big[    | \widetilde{\mathcal{E}_n}(h) - \widetilde{\mathcal{W}}_0(h) | \ge   t n^{-1/2}  (q_{ [ \log_2 (n/t) ]^+  }(h)  + 1)  \Big]  \le c_1 n e^{-c_2 t}   . \]
\end{proposition}

\begin{remark}
\cite[Theorem 3.5]{kol94} is stated in the general setting of a probability space $(X, \mathcal{A}, P)$ that admits a binary expansion, i.e.\ a sequence $\Delta = (\Delta_j)_{j \ge 0}$ of partitions of $X$ satisfying
\begin{itemize}
\item For $i \ge 0$, $\Delta_j = (\Delta_{j,k})_{k = 0, \ldots , 2^j - 1}$ with $\Delta_{j, k} \in \mathcal{A}$ for all $k$;
\item $\Delta_{0,0} = X$, and for all $j \ge 0$ and $k = 0, \ldots, 2^j - 1$,
\[ \Delta_{j, k} = \Delta_{j+1, 2k} \cup \Delta_{j+1, 2k+1}  \quad \text{and} \quad \Delta_{j+1, 2k} \cap \Delta_{j+1, 2k+1} = \emptyset ; \]
\item For $j \ge 0$ and $k = 0, \ldots, 2^j - 1$, $P(\Delta_{j, k}) = 2^{-j}$;
\item $\mathcal{A}$ is the completion with respect to $P$ of the $\sigma$-algebra generated by $\cup_{i \ge 0} \Delta_i$.\end{itemize}
Denote by $\widetilde{\mathcal{E}_n}$ the empirical process on $X$, defined by analogy with \eqref{e:ep}, and by $\widetilde{\mathcal{W}}_0$ the Brownian bridge process on $X$. Then \cite[Theorem 3.5]{kol94} states that there exists $c_1, c_2 > 0$ such that, for every $n \ge 1$ there exists a coupling of $\widetilde{\mathcal{E}_n}$ and $\widetilde{\mathcal{W}}_0$ satisfying, for every measurable $h : X \to \R$ such that $\|h\|_\infty \le 1$, and every $x > 0$ and $y \ge 1$,
\begin{equation}
\label{e:kol}
  \P \Big[   n^{1/2} | \widetilde{\mathcal{E}_n}(h) - \widetilde{\mathcal{W}}_0(h) | \ge  x + x^{1/2} y^{1/2} \big( \hat{q}_{ [ \log_2 (n/y) ]^+  }(h)  + 1 \big) \Big]  \le c_1 ( e^{-c_2 x} + n e^{-c_2 y} )   ,
  \end{equation}
where
\[ \hat{q}_m(g) =  \Big(  \sum_{j = 0}^m   2^j  \sum_{k = 0}^{2^j - 1}  \int_{\Delta_{j, k}} \Big|   h(u) - 2^j \int_{\Delta_{j, k}} h(v) P(dv) \Big|^2 P(du) \Big)^{1/2}  . \]
In the Euclidean setting $(X, \mathcal{A}, P) = (I, \mathcal{B}, dx)$, we have $\hat q_{m}(h)=q_{m}(h)$, hence we obtain Proposition \ref{p:kol} by setting $x = y = t$ in \eqref{e:kol} and adjusting the constants.
\end{remark}

\begin{remark}
The Euclidean setting was also considered in \cite{kol94}, however in deriving the bounds in \eqref{e:qbounds}, \cite{kol94} used the trivial inequality $ | h-\bar h | \leqslant \text{\rm{diam}}(\Delta )\|\nabla h\|_{\infty ,I}$ rather than the Poincar\'e-Wirtinger inequality, which led to a term $\|\nabla h\|_{I,\infty }$ in place of our $\|\nabla h\|_{L^{2}(I)}$.
\end{remark}

Let us make the connection between shot noise fields and empirical processes. Fix $\lambda \ge 1$, and let $ N=\mathcal{P}_{\lambda }(I)$, so that $N \sim \textrm{Pois}(\lambda)$. Then, for $h  \in L^1(I) $,
\[  \sum_{i \in \mathcal{P}_\lambda} h(i)  \stackrel{d}{=} \sum_{i=1}^N h(X_i)  \]
where $(X_i)_{i \ge 1}$ is an i.i.d.\ sequence of random variables independent of $\mathcal  P_{\lambda }$ distributed uniformly on $I$. Hence if we let $\widetilde{\mathcal{E}_N}$ be, conditionally on $N$, an independent copy of the $N^{\text{th}}$ coordinate of the empirical process on $I$, then
\[ (\widetilde{\mathcal{P}_\lambda}(h))_{h \in L^1(I)}  \stackrel{d}{=}  \Big(  \frac{ \sum_{i=1}^N h(X_i) - \lambda \int h}{\sqrt{\lambda}} \Big)_{h \in L^1(I)}= \Big( \frac{ \sqrt{N} \widetilde{\mathcal{E}_N}(h) + (N - \lambda) \int h  }{\sqrt{\lambda}}  \Big)_{h \in L^1(I)} . \]
Similarly, if we let $Z$ denote a standard Gaussian random variable independent of $ \widetilde{\mathcal{W}}_0$ (the Brownian bridge process on $I$), then
\[ (\widetilde{\mathcal{W}}(h))_{h  \in L^2(I)}   \stackrel{d}{=} \Big( \widetilde{\mathcal{W}}_0(h) +  Z \int h \Big)_{h  \in L^2(I)} . \]
Hence every coupling of $(N, \widetilde{\mathcal{E}_N}, Z, \widetilde{\mathcal{W}}_0)$ induces a coupling of $(\widetilde{\mathcal{P}}(h), \widetilde{\mathcal{W}}(h))_{h  \in L^1(I) \cap L^2(I)}$ such that
\begin{align}
\label{e:coup}
 & \sqrt{\lambda} \big( \widetilde{\mathcal{P}}(h) - \widetilde{\mathcal{W}}(h) \big) \\
 \nonumber & \qquad \qquad =  \sqrt{N} \big( \widetilde{\mathcal{E}_N}(h) - \widetilde{\mathcal{W}}_0(h) \big) + (\sqrt{N} - \sqrt{\lambda}) \widetilde{\mathcal{W}}_0(h) + (N - \lambda - \sqrt{\lambda} Z) \int h .
  \end{align}

\smallskip
To help analyse the expression in \eqref{e:coup} we use the following univariate coupling lemma for Poisson random variables, whose proof is deferred to the end of the section:

\begin{lemma}
\label{l:ucl}
There exists universal constants $c_1, c_2 > 0$ such that, for every $\lambda \ge 1$ there is a coupling of $N \sim \textrm{Pois}(\lambda)$ and a standard Gaussian $Z$ such that, for $t \ge 1$,
\[ \P[ | N - \lambda - \sqrt{\lambda} Z| \ge t ] \le  c_1 e^{-c_2 t} .\]
\end{lemma}

We shall also use the following standard Poisson and Gaussian concentration bounds
\begin{equation}
\label{e:cb}
 \P[ | N- \lambda |  \ge t ] \le  2 e^ { - \frac{t^2}{2(\lambda + t)} }  \quad \text{and} \quad  \P[ |Z|   \ge t ]  	\le 2 e^{ -t^2 / 2}
 \end{equation}
 valid for all $\lambda, t> 0$, which can be derived using a standard Chernoff bound
\[ \mathbb{P}[Y \ge t] \le e^{-at} \mathbb{E}[e^{aY}]  \, , \quad \text{for all } a, t > 0   \text{ and random variables } Y, \]
either setting $Y = \pm Z$ and $a = \pm t$ in the Gaussian case, or $Y=\pm (N-\lambda )$ and $a= \pm \log(1  \pm t / \lambda )$ in the Poisson case, using the fact that $(1+x)\log(1+x) - x \ge x^{2}/(2+2x/3))$ for all $x > -1$.

\smallskip
We are now ready to give the proof of Proposition \ref{p:sip}:

\begin{proof}[Proof of Proposition \ref{p:sip}]
 In the proof $c_1, c_2 > 0$ will denote positive constants depending only on the dimension which may change from line to line.
  
\smallskip
We begin by constructing the necessary coupling. Fix $\lambda \ge 1$, let $(Z^i)_{i \in \Z^d}$ be an i.i.d.\ sequence of standard Gaussians, and let $(\widetilde{\mathcal{W}}^i_0)_{i \in \Z^d}$ be a sequence of independent Brownian bridge processes on $i+I$. For each $i \in \Z^d$, let $N^i$ be a $\textrm{Pois}(\lambda)$-distributed random variable coupled to $Z^i$ so as to satisfy the conclusion of Lemma~\ref{l:ucl}, and conditionally on $N^i$, let $\widetilde{\mathcal{E}^i_{N^i}}$ denote the empirical process on $i + I$ coupled to $\widetilde{\mathcal{W}}^i_0$ so as to satisfy the conclusion of Proposition \ref{p:kol}. This defines a coupling of $(N^i, \widetilde{\mathcal{E}^i_{N^i}}, Z^i, \widetilde{\mathcal{W}}_0^i)_{i \in \Z^d}$, which, by \eqref{e:coup} and the decomposition $g = \sum_{i \in \Z^d} g_i$, induces for every $g  \in L^1(\R^d) \cap L^2(\R^d)$ a coupling of $\widetilde{\mathcal{P}_\lambda}(g) = \sum_{i \in \Z^d} \widetilde{\mathcal{P}_\lambda}(g_i)$ and $\widetilde{\mathcal{W}}(g) = \sum_{i \in \Z^d} \widetilde{\mathcal{W}}(g_i)$ such that, for every $i \in \Z^d$,
\begin{align}
\label{e:coup2}
& \nonumber   \sqrt{\lambda}  ( \widetilde{\mathcal{P}_\lambda}(g_i) - \widetilde{\mathcal{W}}(g_i) )   \\
& \qquad = \Big( \sqrt{N^i}( \widetilde{\mathcal{E}^i_{N^i}}(g_i) - \widetilde{\mathcal{W}}^i_0(g_i) ) + (\sqrt{N^i} - \sqrt{\lambda}) \widetilde{\mathcal{W}}^i_0(g) + (N^i - \lambda - \sqrt{\lambda} Z^i) \int g_{i} \Big) .
  \end{align}

\smallskip
Now let $t \ge 1$ be given and abbreviate $m' =  [  \log_2(2\lambda/ t) ]^+$. Then by the union bound and~\eqref{e:coup2}
\begin{align*}
& \P \Big[ \sqrt{\lambda} | \widetilde{\mathcal{P}_\lambda}(g) - \widetilde{\mathcal{W}}(g) |  \ge  3 t \sum_{i \in \Z^d}  (q_{m'}(g_i ) + \|g_i\|_\infty )  \Big]  \\
 &  \quad \le   \sum_{i  \in \Z^d: \|g_i\|_\infty  \neq 0 }  \P \Big[  \sqrt{\lambda} | \widetilde{\mathcal{P}_\lambda}(g_i) - \widetilde{\mathcal{W}}(g_i) | ] \ge 3 t ( q_{m'}(g_i)   + \|g_i\|_\infty )  \Big]    \\
 & \quad  \le |\{ i  \in \Z^d: (i + I) \cap B(r) \neq \emptyset \}|     \\
 & \qquad   \times \max_{i \in \Z^d : \|g_i\|_\infty \neq 0} \Big\{ \P \Big[  \big|  \sqrt{N^i}( \widetilde{\mathcal{E}^i_{N^i}}(g_i) - \widetilde{\mathcal{W}}^i_0(g_i) ) \big| \ge   t ( q_{m'}(g_i ) + \|g_i\|_\infty )  \Big] \\
 & \qquad \qquad   \qquad +  \P  \Big[ \big| (\sqrt{N^i} - \sqrt{\lambda}) \widetilde{\mathcal{W}}^i_0(g) \big| \ge t  \|g_i\|_\infty  ]  \Big]  +  \P\Big[ \big| ( N^i - \lambda - \sqrt{\lambda} Z^i )  \smallint g_i \big| \ge t  \|g_i\|_\infty  \Big]  \Big\}  \\
  & \quad  \le c_1 r^d     \times \max_{i \in \Z^d : \|g_i\|_\infty \neq 0} \Big\{  \sup_{n \le \max\{2\lambda, t\}  } \P \Big[  \sqrt{n} | \widetilde{\mathcal{E}^i_n}(g_i) - \widetilde{\mathcal{W}}_0^i(g_i)| \ge   t ( q_{m'}(g_i ) + \|g_i\|_\infty )  \, \big|  \, N^i = n \Big] \\
 & \qquad \qquad \qquad  \qquad + \P\big[ N^i > \max\{2 \lambda, t\} \big] +  \P  \Big[ \big| \big(\sqrt{N^i} - \sqrt{\lambda} \big)  \widetilde{\mathcal{W}}^i_0(g_i) \big| \ge t  \|g_i\|_\infty  ]  \Big]  \\
 & \qquad \qquad \qquad \qquad \qquad + \P\Big[ \big| N^i - \lambda - \sqrt{\lambda} Z^i \big|  \|g_i\|_{L^1}  \ge t  \|g_i\|_\infty  \Big]  \Big\} .
\end{align*}
We analyse the four summands in the above bound separately. First observe that, since $m \mapsto q_m(h)$ is non-decreasing,
\[  \sup_{n \le \max\{2\lambda, t\}  }  q_{ [ \log_2 (n/t) ]^+  }(h) \le q_{[  \log_2(2\lambda/ t) ]^+}(h) = q_{m'}(h)    . \]
Hence applying Proposition~\ref{p:kol} to the function $h_i = g_i / \|g_i\|_\infty$, for which
\[ q_{m'}(h_i) = q_{m'}(g_i) / \|g_i\|_\infty , \]
we have
\begin{align*}
 &   \sup_{n \le  \max\{ 2\lambda , t \} } \P \big[  \sqrt{n} | \widetilde{\mathcal{E}^i_n}(g_i) - \widetilde{\mathcal{W}}_0^i(g_i)| \ge   t  ( q_{ m'  }(g_i )  + \|g_i\|_\infty )   \, \big|  \, N^i = n \big] \\
   &  \qquad \le  \sup_{n \le  \max\{ 2\lambda , t \} } \P \big[  \sqrt{n} | \widetilde{\mathcal{E}^i_n}(h_i) - \widetilde{\mathcal{W}}_0^i(h_i)|  \ge   t  (  q_{ [ \log_2 (n/t) ]^+  }(h_i)  + 1 )   \, \big|  \, N^i = n \big]  \\
   &  \qquad  \le c_1 \max\{2\lambda, t\} e^{-c_2 t}  \le c_1 \lambda e^{-c_2 t} .
   \end{align*}
  Next we have
\begin{align*}
 \P  \Big[ \big| \big(\sqrt{N^i} - \sqrt{\lambda} \big)  \widetilde{\mathcal{W}}^i_0(g_i) \big| \ge t \|g_i\|_\infty  \Big]  & \le  \P \Big[  | \sqrt{N^i} - \sqrt{\lambda} | \ge \sqrt{  t} \Big] + \P \big[ |\widetilde{\mathcal{W}}^i_0(g_i) | \ge \sqrt{t}  \|g_i\|_\infty \big] \\
&  \le   \P[  | \sqrt{N} - \sqrt{ \lambda}| \ge \sqrt{  t} ] + \P[ |Z| \ge \sqrt{t} ]  ,
\end{align*}
where in the last inequality we used that $\E[ \widetilde{\mathcal{W}}_0^i(g_i) ^2] = \int g_i^2 - (\int g_i)^2 \le \|g_i\|^2_\infty$. Using the fact that, for $a, b, c > 0$,
\[  |\sqrt{a} - \sqrt{b}| \ge \sqrt{c} \quad  \Longrightarrow \quad |a-b| \ge \max\{ \sqrt{bc}, c \} , \]
and applying the concentration bounds in \eqref{e:cb} we have
\[    \P[   | \sqrt{N} -  \sqrt{ \lambda } | \ge \sqrt{ t} ]  + \P[ |Z| \ge \sqrt{t} ]   \le    \P[   |N  -   \lambda  | \ge \max\{ \sqrt{ \lambda t}, t \} ]  + \P[ |Z| \ge \sqrt{t} ]   \le c_1 e^{-c_2 t} . \]
Finally
\[ \P[ N^i \ge \max\{2 \lambda, t \}  ]  \le  \P[ |N - \lambda | \ge   \max\{ \lambda, t - \lambda \}  ]   \le  c_1 e^{- c_2 \max\{ \lambda, t\} } , \]
and
\[  \P\Big[ \big|N^i - \lambda - \sqrt{\lambda} Z^i \big|  \|g_i\|_{L^1}  \ge \|g_i\|_\infty t \Big]   \le  \P \Big[ \big| N^i - \lambda - \sqrt{\lambda} Z^i \big|  \ge  t  \Big]  \le  c_1 e^{-c_2  t } ,\]
where we used the concentration bounds in \eqref{e:cb} and Lemma \ref{l:ucl}. Combining the above we deduce that
\[  \P \Big[ \lambda^{1/2} | \widetilde{\mathcal{P}_\lambda}(g) - \widetilde{\mathcal{W}}(g) |  \ge  3 t \sum_{i \in \Z^d}  ( q_{ m'  } (g_i )  + \|g_i\|_\infty )  \Big]  \le  c_1 r^d  \lambda e^{-c_2 t}  , \]
and we obtain the result by adjusting constants.

\end{proof}

\subsection{Strong invariance principle for shot noise fields; proof of Theorem \ref{t:sip}}
We are now ready to prove Theorem~\ref{t:sip}. Let us first prove (1). In the proof $c_1, c_2 > 0$ will denote positive constants depending only on $\beta$ and $d$ which may change from line to line. Let $\lambda \ge 1$ be given and recall the coupling between $\widetilde{\mathcal{P}_\lambda}$ and $\widetilde{\mathcal{W}}$ in Proposition~\ref{p:sip}; we shall use the same coupling (which is in fact independent of the kernel $g$). Recall also the definition of $\sigma_d(\lambda)$ in \eqref{e:sigma}. Fix $R, t \ge 1$ and parameters $r, \varepsilon > 0$ to be chosen later. Then by first using that
\[  \| f_\lambda - f \|_{\infty, B(R)} \le \| f_\lambda - f\|_{\infty, \varepsilon \Z^d \cap B(R) }  + \sqrt{d} \varepsilon \| f_\lambda \|_{C^1(B(R))} +  \sqrt{d} \varepsilon \| f \|_{C^1(B(R))} , \]
and then approximating $f_\lambda$ and $f$ by the $r$-dependent fields $f^r_\lambda$ and $f^r$ respectively, we have by the union bound,
\begin{align}
  \nonumber &\P \Big[  \| f_\lambda - f \|_{\infty, B(R)}  \ge (3 +  2\sqrt{d}) t \sigma_d(\lambda)  \Big]   \\
 \nonumber & \qquad \le   \P \Big[  \| f^r_\lambda - f^r \|_{\infty, \varepsilon \Z^d \cap B(R)}  \ge  t \sigma_d(\lambda)  \Big]  \\
\nonumber  & \qquad \quad +   \P \Big[  \| f_\lambda - f^r_\lambda \|_{\infty, B(R)}  \ge t  \sigma_d(\lambda)  \Big]  +   \P \Big[  \| f - f^r \|_{\infty, B(R)}  \ge t \sigma_d(\lambda) \Big]   \\
\nonumber  & \qquad \quad \quad +   \P \Big[  \| f_\lambda \|_{C^1(B(R))}  \ge   t  \varepsilon^{-1} \sigma_d(\lambda)  \Big]  +   \P \Big[  \|f \|_{C^1(B(R))}  \ge  t \varepsilon^{-1} \sigma_d(\lambda)  \Big]  \\
\label{e:psip0} &  \qquad \quad =: E_1 + E_2 + E_3 + E_4 + E_5 .
  \end{align}
  We next substitute parameters
  \begin{equation}
  \label{e:epsr}
   \varepsilon = \min\{t^{-1}, \lambda^{-1/2} \}   \qquad \text{and} \qquad r = (t \lambda)^{1/(\beta-d)}
   \end{equation}
  and control each of the summands in the right-hand side of \eqref{e:psip0} in turn (see Remark \ref{r:error} for some comments on this choice).

  \smallskip
For $i \in \Z^d$, recall the translated cube $C_i = i + I$, and define $g_i = g^r|_{C_i}$. Recall the definition of $q_m(h)$ in \eqref{e:q}. Since $g \in C^1(\R^d)$, using the bounds in \eqref{e:qbounds} we have, for any $i \in \Z^d$, 
\[ q_m(g_i) \le     c_1 \|\nabla g_i\|_{L^{2}(i + I)}  \times \begin{cases}   1 &  d = 1 , \\   \sqrt{m+1} & d = 2 ,  \\   2^{(1/2-1/d)m} & d \ge 3  . \end{cases}   \]
By the decay of $g$ in Assumption \ref{a:2}, we have the following bounds (uniformly over $r \ge 1$),
\[   \sum_{i \in \Z^d} \|\nabla g_i\|_{L^{2}(i + I)} =\|\nabla g\|_{L^{2}(\mathbb{R}^{d})}  \le c_1  \quad \text{and} \quad  \sum_{i \in \Z^d} \|g_i\|_{\infty} \le c_1  \]
and so, considering the definition of $\sigma_d(\lambda)$ in \eqref{e:sigma}, we obtain
 \begin{align*}  \lambda^{-1/2}  \sum_{i \in \Z^d}  \big( q_{[ \log_2(2\lambda/t) ]^+ }(g_i)  + \|g_i\|_\infty \big) & \le   c_1  \lambda^{-1/2}  \begin{cases}   1 &  d = 1 , \\   \sqrt{ \log \lambda} & d = 2,   \\  \lambda^{1/2-1/d}  & d \ge 3  , \end{cases} \\
 &  = c_1  \sigma_d(\lambda)  .
  \end{align*}
Applying Proposition \ref{p:sip} to the functions $g^r(\cdot - i)$ for $i \in \varepsilon \Z^d \cap B(R)$, by the union bound we obtain
  \begin{equation}
  \label{e:psip1}
    E_1 = \P \Big[  \| f^r_\lambda - f^r \|_{\infty, \varepsilon \Z^d \cap B(R)}  \ge  t \sigma_d(\lambda)  \Big]    \le c_1 \varepsilon^{-d} R^d r^d \lambda e^{-c_2 t } .
     \end{equation}

 \smallskip
We move on to estimating the remaining terms $E_i$, $i=2,3,4,5$. Recalling the definition of $r$ in \eqref{e:epsr} we observe that
\[ t \sigma_d(\lambda) \ge t \lambda^{-1/2} =  t^2 \lambda^{1/2}  r^{d-\beta} . \]
 Hence applying Proposition \ref{p:snsup} (to $u = t^2$) and the union bound
 \[ E_2 = \P \Big[  \| f_\lambda - f^r_\lambda \|_{\infty, B(R)}  \ge t  \sigma_d(\lambda)  \Big]  \le \P \Big[  \| f_\lambda - f^r_\lambda \|_{\infty, B(R)}  \ge t^2 \lambda^{1/2}  r^{d-\beta }   \Big]   \le c_1 R^d  \lambda e^{- c_2   t } . \]
Similarly
 \[ t \sigma_d(\lambda) \ge t \lambda^{-1/2}  =    t^2 \lambda^{1/2}  r^{d-\beta}  \ge   \sqrt{t} r^{d/2-\beta}  , \]
 and hence applying Proposition \ref{p:gfsup} (to $u = \sqrt{t} $) and the union bound
    \begin{equation}
  \label{e:psip2}
  E_3 = \P \Big[  \| f - f^r \|_{\infty, B(R)}  \ge t \sigma_d(\lambda) \Big]  \le \P \Big[  \| f - f^r \|_{\infty, B(R)}  \ge \sqrt{t} r^{d/2-\beta}  \Big]   \le  c_1 R^d e^{-c_2 t } .
  \end{equation}

  \smallskip
 Finally recalling the definition of $\varepsilon$ in \eqref{e:epsr} we have
 \[   t  \varepsilon^{-1} \sigma_d(\lambda)  \ge t \varepsilon^{-1} \lambda^{-1/2}  \ge \max\{t^2 \lambda^{-1/2}  , \sqrt{t} \}    \]
  and so by Propositions \ref{p:sndb} and \ref{p:gfdb} and the union bound
   \begin{equation}
  \label{e:psip3}
 E_4 =   \P \Big[  \| f_\lambda \|_{C^1(B(R))}  \ge   t  \varepsilon^{-1} \sigma_d(\lambda)  \Big]   \le  \P \Big[  \| f_\lambda \|_{C^1(B(R))}  \ge   t^2 \lambda^{-1/2}    \Big]   \le c_1 R^d \lambda e^{- c_2 t  }
    \end{equation}
 and
     \begin{equation}
  \label{e:psip4}
  E_5 = \P \Big[  \| f \|_{C^1(B(R))}  \ge   t  \varepsilon^{-1} \sigma_d(\lambda)  \Big]   \le  \P \Big[  \| f \|_{C^1(B(R))}  \ge   \sqrt{t}   \Big]   \le c_1 R^d e^{- c_2 t } .
   \end{equation}

   \smallskip
   Putting \eqref{e:psip1}--\eqref{e:psip4} into \eqref{e:psip0} we have, for sufficiently large $t$,
   \[ \P \Big[  \| f_\lambda - f \|_{\infty, B(R)}  \ge (3 + 2 \sqrt{d}) t \sigma_d(\lambda)  \Big]   \le  \sum_{i = 1}^5 E_i \le c_1 \varepsilon^{-d} R^d r^d \lambda e^{-c_2 t } . \]
Recalling the definition of the parameters $\varepsilon$ and $r$ in \eqref{e:epsr}, and also the definition of $c_{d, \beta}$ in \eqref{e:sigma}, the right-hand side of the above display equals
     \[   c_1 R^d t^d \max\{ t^d,  \lambda^{d/2} \}  ( \lambda t)^{ d/(\beta-d)}  \lambda e^{-c_2 t } \le  c_1 R^d \lambda^{c_{d, \beta}}  t^{d+ d/(\beta-d)}  e^{-c_2 t }   , \] 
  which implies the result by adjusting constants.
  
 \smallskip 
The proof of (2) is a straightforward extension. The key observations are that (i) by the convolution representations in Section \ref{ss:conrep} (the exchange of integration and differentiation is justified by the dominated convergence theorem, in light of the decay in Assumption \ref{a:2})
\[ \partial^{\alpha }f_{\lambda } = ( \partial^\alpha g) \star \widetilde{P_\lambda} \quad \text{and} \quad \partial^\alpha f = (\partial^\alpha g) \star \widetilde{W} ,\]
and (ii) the coupling in the proof of (1) is chosen independently of the kernel $g$. Hence by applying (1) to $\partial ^{\alpha }g$ instead of $g$ for each $ \alpha \in I$, the result follows by the union bound.

\begin{remark}
\label{r:error}
{ 
The main contribution to the right-hand side of \eqref{e:psip0} is from $E_1$, and we make the choice of parameters in \eqref{e:epsr} to balance $E_1$ with the contribution from the other summands $E_i$, $i \ge 2$, controlling the truncation and discretisation errors, uniformly over $\lambda$ and $t$. Note however that this choice is not optimal for all $\lambda$ and $t$ simultaneously, and the bounds on $E_i$, $i \ge 2$, could be sharpened in some cases.}
   \end{remark}

\subsection{Proof of the univariate coupling lemma for Poisson random variables}

In this section we obtain Lemma \ref{l:ucl} by applying the following `Tusn\'{a}dy-type' lemma on coupling a simple random walk to Brownian motion:

\begin{lemma}[{\cite[Lemma 1 and Section 6]{kmt75}}]
\label{l:kmtucl}
Let $(X_i)_{i \ge 1}$ be a sequence of i.i.d.\ random variables with zero mean, unit variance, and a finite moment generating function in a neighbourhood of the origin. Then there are constants $\varepsilon, c_1, c_2 > 0$ such that, for every $n \ge 1$, there is a coupling of $S_n = \sum_{i=1}^n X_i$ and a standard Gaussian $Z$ satisfying
\[    |S_n - \sqrt{n} Z | \le \frac{c_1 S_n^2 }{n} + c_2 \quad   \text{if} \quad |S_n| < \varepsilon n .     \]
\end{lemma}

\begin{proof}[Proof of Lemma \ref{l:ucl}]
Let $(X_i)_{i \ge 1}$ be a sequence of i.i.d.\ random variables distributed as $X_i \stackrel{d}{=} \textrm{Pois}(1) - 1$, and let $\varepsilon, c_1, c_2 > 0$ be the constants appearing in Lemma \ref{l:kmtucl} for this $X_i$.  By adjusting constants it suffices to prove the result for $t$ sufficiently large, and we split the proof into two cases: (i) $4 \le t \le   \varepsilon \lfloor \lambda \rfloor $, and (ii) $t \ge  \varepsilon \lfloor \lambda \rfloor $.

\smallskip
For the case $4 \le t \le   \varepsilon \lfloor \lambda \rfloor $, note that $N \sim \textrm{Pois}(\lambda) \stackrel{d}{=}  \sum_{i = 1}^{\lfloor \lambda \rfloor} X_i + \lfloor \lambda \rfloor + Y$, where $Y$ is an independent random variable distributed as $\textrm{Pois}( \lambda - \lfloor \lambda \rfloor)$. Consider the coupling of $ \sum_{i = 1}^{\lfloor \lambda \rfloor} X_i $ and $Z$ guaranteed by Lemma \ref{l:kmtucl} and define $N =   \sum_{i = 1}^{\lfloor \lambda \rfloor} X_i + \lfloor \lambda \rfloor + Y$. Then since
 \[ N - \lambda - \sqrt{\lambda} Z   = \Big( \sum_{i = 1}^{\lfloor \lambda \rfloor} X_i  -  \sqrt{ \lfloor \lambda \rfloor } Z \Big) + Y + ( \lfloor \lambda \rfloor - \lambda ) + (\sqrt{\lfloor \lambda \rfloor} - \sqrt{\lambda} ) Z , \]
 we have
\begin{align*}
&  \mathbb{P}[ |N - \lambda - \sqrt{\lambda} Z | \ge t ]  \le  \P \Big[  \Big|   \sum_{i = 1}^{\lfloor \lambda \rfloor} X_i  - \sqrt{ \lfloor \lambda \rfloor} Z \Big|  \ge t/4 \Big]  + \P[ Y \ge t/4 ] + \P[ |Z| \ge t /4 ] \\
 & \le \mathbb{P} \Big[  \Big|   \sum_{i = 1}^{\lfloor \lambda \rfloor} X_i  - \sqrt{ \lfloor \lambda \rfloor} Z \Big|  \ge t/4 ,   \sum_{i = 1}^{\lfloor \lambda \rfloor} X_i  < \varepsilon \lfloor \lambda \rfloor  \Big]  +  \P \Big[ \sum_{i = 1}^{\lfloor \lambda \rfloor} X_i \ge \varepsilon  \lfloor \lambda \rfloor    \Big]  \\
 & \qquad \qquad \qquad + \P[ Y \ge t/4 ] + \P[ |Z| \ge t/4 ]  \\
 &    \le \mathbb{P} \Big[  \Big|   \frac{c_1 (  \sum_{i = 1}^{\lfloor \lambda \rfloor} X_i  ) ^2 }{ \lfloor \lambda \rfloor } + c_2   \Big|  \ge t/ 4 \Big]  +  \P \Big[ \sum_{i = 1}^{\lfloor \lambda \rfloor} X_i \ge \varepsilon  \lfloor \lambda \rfloor    \Big]  + \P[ Y \ge t/4 ] + \P[ |Z| \ge t/4 ]  \\
 &    \le 2 \mathbb{P} \Big[   \Big| \textrm{Pois}(\lfloor \lambda \rfloor ) - \lfloor \lambda \rfloor   \Big|  \ge \min\Big\{   \sqrt{ \lfloor \lambda \rfloor / c_1  } \sqrt{t/4 - c_2 } ,  \varepsilon  \lfloor \lambda \rfloor \Big\} \Big]  \! + \P[ \textrm{Pois}(1) \ge t/4 ] + \P[ |Z| \ge t/4 ]  
  \end{align*}
where in the first inequality we used $| \lambda -  \lfloor \lambda \rfloor | \le 1 \le t/4$ and $|\sqrt{\lambda} - \sqrt{ \lfloor \lambda \rfloor } | \le 1$, in the third inequality we used Lemma \ref{l:kmtucl}, and in the last inequality we used that $ \sum_{i = 1}^{\lfloor \lambda \rfloor} X_i \sim \textrm{Pois}(\lfloor \lambda \rfloor ) - \lfloor \lambda \rfloor$ and that $Y$ is stochastically dominated by $\textrm{Pois}(1)$. The concentration bounds in \eqref{e:cb} imply that the above is at most $c_3 e^{-c_4 \min\{t,   \lambda  \}}$ for some $c_3, c_4 > 0$ and $t$ sufficiently large, which gives the result in this case by taking $t$ sufficient large and adjusting constants.

\smallskip
In the case $t \ge  \varepsilon \lfloor \lambda \rfloor $, we instead bound
\[    \P[ |N - \lambda - \sqrt{\lambda} Z | \ge t ]  \le  \P[ |N-\lambda| \ge t / 2 ] + \P[ |Z| \ge  t / (2 \sqrt{\lambda} )  ]  . \]
Again the concentration bounds in \eqref{e:cb} imply that the above is at most $c_5 e^{- c_6 \min\{t , t^2 / \lambda \}}$, which gives the result also in this case after adjusting constants.
\end{proof}

\medskip
\section{Continuity of the critical level}
\label{s:cont}

In this section we prove the quantitative continuity result for sequences of stationary planar random fields stated in Proposition \ref{p:cont} above.

\subsection{Finite-size criterion}
\label{s:fsc}
We first establish a finite-size criterion that applies to a single planar random field. This states roughly that, as soon as crossing events occur with sufficiently small (resp.\ large) probability at a given level, one can deduce that the critical level is not much smaller (resp.\ bigger) than this level. The proof is similar to classical bootstrapping arguments for Bernoulli percolation \cite[Section 5.1]{kes82}, although in our setting the sprinkled decoupling property replaces independence (see also \cite{pr15,mv20} for similar arguments).

\smallskip
Recall that $\textrm{Cross}_\ell[a, b]$ (resp.\ $\textrm{Cross}^{tb}_\ell[a, b]$) denotes the event that $\{f \le \ell\}$ crosses the rectangle $[0, a] \times [0, b]$  from left to right (resp.\ top to bottom).

\begin{proposition}[Finite-size criterion]
\label{p:fsc}
Let $f$ be a stationary planar regular random field that satisfies the polynomial sprinkled decoupling property with constants $(c_h, c_{e,1}, c_{e,2})$. Then there exists $c > 0$ depending only on $c_h$, and $\delta, R_0 > 0$ depending only on $(c_{e,1}, c_{e,2})$, such that, for all $\ell \in \R$ and $R \ge R_0$,
\[ \max\big\{ \mathbb{P}[  \textrm{Cross}_\ell[R, 3R] ] , \mathbb{P}[  \textrm{Cross}^{tb}_\ell[3R, R] ]  \big\} \le \delta \ \  \quad \quad \Longrightarrow \quad  \ell_c(f) \ge \ell -  c R^{-c_h} \ \  \]
and
\[ \min\big\{  \mathbb{P}[  \textrm{Cross}_\ell[3 R, R] ] ,  \mathbb{P}[  \textrm{Cross}^{tb}_\ell[R, 3R] ]  \big\} \ge 1 - \delta \quad \Longrightarrow \quad  \ell_c(f) \le \ell + c R^{-c_h}   . \]
\end{proposition}
\begin{proof}
Fix $\delta > 0$ and $R_0 \ge 1$ such that $49 (\delta + \sqrt{c_{e,1}}   R_0^{-c_{e,2}/2} ) < 1$. By the classical construction of Kesten \cite[Section 5.1]{kes82} (see Figure \ref{f:bootstrap}), for every $R \ge 1$ one can find collections $(A_i)_{i=1,2,\ldots ,7}$ and $(B_j)_{j=1,2,\ldots ,7}$ of translated copies of the events $ \textrm{Cross}_\ell[R, 3R]$ or $ \textrm{Cross}^{tb}_\ell[3R, R]$ such that the following holds for every $\ell \in \R$:
\begin{itemize}
\item $\textrm{Cross}_\ell[3R, 9R]$ implies $\cup_{i,j = 1, \ldots , 7} (A_i\cap B_j)$; and
\item For each $i ,j= 1, \ldots, 7$, the distance between the rectangles associated with $A_i$ and $B_j$ is at least $R$.
\end{itemize}
Then the union bound, sprinkled decoupling, and stationarity imply that, for every $R \ge 1$ and $\ell \in \R$,
\begin{align}
\label{e:sd1}
 \nonumber \P[  \textrm{Cross}_\ell[3R, 9R] ] &  \le 49 \max_{i,j= 1, 2,\ldots, 7}  \P[ f \in  A_i \cap B_j ] \\
 \nonumber & \le  49 \max_{i ,j= 1, 2,\ldots, 7}  \P[ f + h(R) \in A_i] \P[  f + h(R) \in B_j ] + e(R) \\
  & \le 49 \max\{ \P[ \textrm{Cross}_{\ell + R^{-c_h}}[R, 3R] ] , \mathbb{P}[  \textrm{Cross}^{tb}_{\ell + R^{-c_h}}[3R, R] ]  \big\}^2 + c_{e,1} R^{-c_{e,2}} .
\end{align}

\begin{figure}
  \includegraphics[scale=0.55]{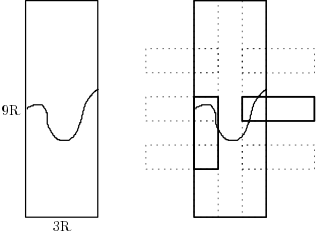}
  \caption{On the left, a realisation of $\cross_\ell[3R, 9R]$. On the right, the $14$ translated copies of $[R, 3R]$ and $[3R, R]$ on which the events $(A_i)_{i=1,\ldots,7}$ and $(B_j)_{j=1,\ldots,7}$ are supported, and in bold a pair $(A_i, B_j)$ such that $A_i \cap B_j$ occurs.}
  \label{f:bootstrap}
\end{figure}

Now suppose $R \ge R_0$ and $\ell \in \R$ are such that
\[\max\big\{ \mathbb{P}[  \textrm{Cross}_\ell[R, 3R] ] , \mathbb{P}[  \textrm{Cross}^{tb}_\ell[3R, R] ]  \big\} \le \delta. \]
For $n \ge 0$, define the sequences $r_n = R 3^n$ and $\ell_n = \ell - \sum_{i=0}^{n-1} r_i^{-c_h}$ with $\ell_0 = \ell$. Abbreviating
\[ a_n =  \max\big\{  \P[  \textrm{Cross}_{\ell_n}[r_n, 3r_n]] ,  \mathbb{P}[   \textrm{Cross}^{tb}_{\ell_n}[3r_n, r_n] ]  \big\}    , \]
 we deduce from \eqref{e:sd1} that, for every $n \ge 0$,
\begin{equation}
\label{e:fs1}
 a_{n+1} \le  49 a_n^2 +  c_{e,1} r_n^{-c_{e,2} }  .
 \end{equation}
Defining $a_n' = 49(a_n + \sqrt{ c_{e,1} }  r_n^{-c_{e,2}/2} ) \ge a_n$, one checks from \eqref{e:fs1} that
\begin{align*}
 a'_{n+1} &= 49(a_{n+1} + \sqrt{ c_{e,1} } r_n^{-c_{e,2}/2} )   \le 49( 49 a_n^2 + c_{e,1} r_n^{-c_{e,2}}  + \sqrt{ c_{e,1} } r_n^{-c_{e,2}/2}) \\
 & \le 49^2 ( a_n + \sqrt{ c_{e,1} }   r_n^{-c_{e,2}/2} )^2  = (a_n')^2 .
 \end{align*}
 Since also
 \[ a_0' = 49(a_0 + \sqrt{ c_{e,1} }  R^{-c_{e,2}/2} )  \le 49(\delta + \sqrt{ c_{e,1} }  R_0^{-c_{e,2}/2} ) < 1 \]
 by the choice of $\delta$ and $R_0$, we deduce that $a_n \le a_n'$ decays exponentially in $n$. Defining $\ell' =  \ell - c R^{-c_h}$, where $c = \sum_{i=0}^\infty 3^{-c_h i} \in (0, \infty)$, so that $\ell_n \downarrow \ell'$, by monotonicity we conclude that
  \begin{equation}
 \label{e:expdecay}
      \max\big\{  \P[  \textrm{Cross}_{\ell'}[r_n, 3r_n]] ,  \mathbb{P}[   \textrm{Cross}^{tb}_{\ell'}[3r_n, r_n] ]  \big\} \le c_1 e^{-c_2 n }   
      \end{equation}
 for all $n \ge 0$ and some $c_1, c_2 > 0$.

 \smallskip 
 To finish the proof of the first statement, define the event $A_n$ that $\{f \le \ell'\}$ contains a path that intersects both $[-r_n/2, r_n/2]^2$ and $\partial [-3r_n/2, 3r_n/2]^2$. By the union bound and stationarity
 \[  \P[A_n] \le   4  \max\big\{  \P[  \textrm{Cross}_{\ell'}[r_n, 3r_n]] ,  \mathbb{P}[   \textrm{Cross}^{tb}_{\ell'}[3r_n, r_n] ]  \big\} \to 0 \]
 as $n \to \infty$. Since $\ell' > \ell_c(f)$ would imply that $\liminf_{n \to \infty} \P[A_n] > 0$, we deduce that $\ell_c(f) \ge \ell' = \ell - c R^{-c_h}$ as required.

\smallskip
For the second statement, by replacing $f \mapsto -f$ and using the regularity of the field, the same reasoning that led to \eqref{e:expdecay} shows that if $R \ge R_0$ and $\ell \in \R$ are such that
 \[ \min\big\{  \mathbb{P}[  \textrm{Cross}_\ell[3 R, R] ] ,  \mathbb{P}[  \textrm{Cross}^{tb}_\ell[R, 3R] ]  \big\} \ge 1 - \delta \]
 then, defining $r_n = R 3^n$ and $\ell'  =  \ell + c R^{-c_h}$, 
\[ \min \big\{  \mathbb{P}[ f \in \textrm{Cross}_{\ell'}[3r_n, r_n] ] , \mathbb{P}[   f \in \textrm{Cross}^{tb}_{\ell'}[r_n, 3r_n] ]  \big\} \ge 1 - c_1 e^{-c_2 n}   \]
 for all $n \ge 0$. By a classical construction (see Figure \ref{f:glue}) one can define events $(A_n)_{n \ge 0}$ satisfying:
 \begin{itemize}
\item If $n$ is even (resp.\ odd) then $A_n$ is a translation of $\cross_{\ell'}[3 r_n, r_n]$ (resp.\ $\cross^{tb}_{\ell'}[r_n, 3r_n]$);
\item For any $n_0 \ge 0$, $\cap_{n \ge n_0}A_n$ implies the existence of an infinite path in $\{f \le \ell' \}$.
\end{itemize}
Then by the Borel-Cantelli lemma $\{f \le  \ell'\}$ contains an infinite component almost surely, and thus $\ell_c(f) \le \ell' = \ell + cR^{-c_h}$ as required.

\begin{figure}
  \includegraphics[scale=0.3]{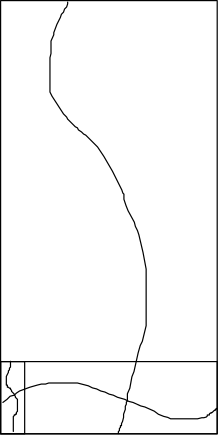}
  \caption{A depiction of the events $A_n$ for $n=1,2,3$.}
  \label{f:glue}
\end{figure}

\end{proof}

\subsection{Proof of Proposition \ref{p:cont}}
Let $s(n), u(n), R(n)$ and $w(R)$ be as in the statement of the proposition, and let $R_0, \delta > 0$ be the constants defined by Proposition \ref{p:fsc} applied to the fields $(s(n) f_n)_{n \ge 1}$; by the uniform sprinkled decoupling property these can be chosen uniformly over $n$. By \eqref{e:c}, \eqref{e:un}, and the definition of the near-critical window, we have
 \begin{align*}
 & \min\{ \mathbb{P}[ f_n \in \textrm{Cross}_{w( R(n) ) + u(n)}[3 R(n), R(n)] ]  , \mathbb{P}[ f_n \in \textrm{Cross}^{tb}_{w( R(n) ) + u(n)}[R(n), 3R(n)] ]  \} \\
 &   \qquad \ge  \min\{ \mathbb{P}[ f \in \textrm{Cross}_{w( R(n) ) }[3 R(n), R(n)] ]  , \mathbb{P}[ f \in \textrm{Cross}^{tb}_{w( R(n) )  }[R(n), 3R(n)] ]  \}  \\
 & \qquad \qquad \qquad -   \P[ \| f_n - f\|_{\infty, B(2R(n)) } \ge u(n)]  \\
 & \qquad \to 1 
 \end{align*}
 as $n \to \infty$. Via a rescaling, this implies that
  \begin{align*}
 &    \min\{ \mathbb{P}[s(n) f_n \in \textrm{Cross}_{s(n)(w( R(n) ) + u(n))}[3 R(n), R(n)] ]  , \mathbb{P}[ s(n) f_n \in \textrm{Cross}^{tb}_{s(n)(w( R(n) ) + u(n))}[R(n), 3R(n)] ]  \} \\
 & \qquad \ge  1 - \delta   
 \end{align*}
for sufficiently large $n$. Hence taking $n$ sufficiently large so that also $ R(n) \ge R_0$, by the finite-size criterion in Proposition~\ref{p:fsc} applied to $s(n) f_n$ we deduce that
 \[ \ell_c( s(n) f_n) \le  s(n) ( w( R(n) ) + u(n) )  + c R(n)^{-c_h}    \]
 where $c > 0$ depends only on $c_h$. Via rescaling, this implies that
  \[ \ell_c(  f_n) \le    w( R(n) ) + u(n)   + c s(n)^{-1} R(n)^{-c_h}   \]
  for sufficiently large $n$. An analogous argument for the events 
  \[ \textrm{Cross}_{ - w( R(n) ) - u(n)  }[ R(n), 3R(n)  ]  \quad \text{and} \quad \textrm{Cross}^{tb}_{ - w( R(n) ) - u(n)  }[ 3R(n), R(n)  ] \] 
  shows that
  \[ \ell_c(f_n) \ge   - w( R(n) ) - u(n)  - c s(n)^{-1} R(n)^{-c_h}  \]
   for sufficiently large $n$, which completes the proof.
\bigskip

\bibliographystyle{halpha-abbrv}
\bibliography{sn}

\end{document}